\documentclass{amsart}
\pdfoutput=1

\usepackage{graphicx}
\usepackage{mathptmx}    
\usepackage{latexsym}
\usepackage{amsfonts}
\usepackage{amsmath}

\newtheorem{thm}{Theorem}[section]

\newtheorem{lem}[thm]{Lemma}

\theoremstyle{definition}
\newtheorem{defn}{Definition}[section]

\begin{document}
\title{A Problem of Erd\"{o}s Concerning Lattice Cubes}
\author{Chengcheng Yang}
\address{Department of Mathematics, Rice University, 6100 Main Street, Houston TX 77005d}
\email{cy2@rice.edu} 
\thanks{The author was support in part by NSF Award\#1745670}
\email{cy2@rice.edu} 
\date{\today}
%\copyrightyear{2008--2020}
%\copyrightnote{\textcopyright~2008--2020 Boris Veytsman}

\begin{abstract}
This paper studies a problem of Erd\"{o}s concerning lattice cubes. Given an $N \times N \times N$ lattice cube, we want to find the maximum number of vertices one can select so that no eight corners of a rectangular box are chosen simultaneously. Erd\"{o}s conjectured that it has a sharp upper bound, which is $O(N^{11/4})$, but no example that large has been found yet. We start approaching this question for small $N$ using the method of exhaustion, and we find that there is not necessarily a unique maximal set of vertices (counting all possible symmetries). Next, we study an equivalent two-dimensional version of this problem looking for patterns that might be useful for generalizing to the three-dimensional case. Since an $n \times n$ grid is also an $n \times n$ matrix, we rephrase and generalize the original question to: what is the minimum number $\alpha(k,n)$ of vertices one can put in an $n \times n$ matrix with entries 0 and 1, such that every $k \times k$ minor contains at least one entry of 1, for $1 \leq k \leq n$? We discover some interesting formulas and asymptotic patterns that shed new light on the question.
\end{abstract}
\maketitle

\section{Motivation}
In 2009, I learned this problem from Dr. Katz in my summer REU program at Indiana University. The information and ideas in this section came completely from him. After the motivation, come my ideas and efforts. Now let's start. 

In the paper {\it Multidimensional Van-der-Corput and sublevel set estimates} \cite{1}, in connection with estimates of Fourier integral operators, the following question was raised. Let 
%\[...\] is more convenient than begin{equation*} ... end{equation*}
\[Q = [0, 1]^n \subset \mathbb{R}^n, \] 
be the unit cube. Let $f$ be a $C^n$ function on $Q$ satisfying
\[\frac{\partial^n f}{\partial x_1 \partial x_2 \ldots \partial x_n } > \Lambda. (\Lambda > 0)  \]
What upper bound can be given on the measure of the set
\[S = \{x \in Q : |f(x)| < 1\}.\]
To solve this problem, we first notice that there is no large n-dimentional  box contained in $S$. 
More precisely, pick two different points in the $x_1$ direction, call them $x_{11}$ and $x_{12}$ with $x_{12} > x_{11}$; 
pick two different points in the $x_2$ direction, call them $x_{21}$ and $x_{22}$ with $x_{22} > x_{21}$; and so on. Then the set of points in the form 
$(x_{1i_1}, x_{2i_2}, \ldots, x_{ni_n})$ are the vertices of an n-dimensional box in $\mathbb{R}^n$, where $i_1, i_2, \ldots, i_n \in \{1,2\}$. 
There is a constant $C$ such that if the volume of this box is bigger than $C$, then $S$ does not contain all points of the box. 
For example, when $n=1$, by the fundamental theorem of calculus, \[
f(x_{12})=f(x_{11}) + \int_{x_{11}}^{x_{12}} \frac{\partial f}{\partial x_1} dx_1.\]
If $x_{12} - x_{11} >  \frac{2}{\Lambda}$ , then \[
2 > |f(x_{12}) - f(x_{11})| > (x_{12} - x_{11}) \Lambda,\] which is a contradiction. So $C = \frac{2}{\Lambda}$. In general, one can show that $C = \frac{2^n}{\Lambda}$ by Fubini's theorem. 

We can make a discrete analogue of the question, which helps to solve the orginal question. 
That is, given an $N \times N \times N$ grid, what is the maximum number of vertices that we can select which does not contain the eight corners of any rectangular box? First of all, we can apply the Cauchy-Schwarz inequality to get an upper bound. 

Let $A$ be a maximal set of vertices that satisfies the condition, and we define the following function: 
\[ I : \{ (i, j, k) | 1 \leq i, j, k \leq N\} \longrightarrow \{0, 1\}, \] where
 \[
I(i, j, k) = \left\{ \begin{array}{rcl}
1 & \mbox{ if } & (i, j, k) \in A, \\
0 & \mbox{ if } & (i, j, k) \notin A. 
\end{array} \right.
\]Here each vertex is indexed by $(i, j, k)$, and we also write $I(i, j, k)$ as $I _{ijk}$ for convenience. 

\begin{lem}
\label{lem:1}
The size of $A$ is at most $O(N^{\frac{11}{4}})$. 
\end{lem}

\begin{proof}
The proof uses the Cauchy-Schwarz inequality several times as shown below.
\[
\begin{array}{rcl}
|A| & = & \sum_{i} \sum_{j} \sum_{k} I_{ijk}   \vspace{0.3cm} \\
     & = & \sum_i (\sum_{j, k} I_{ijk}) \times 1  \vspace{0.3cm} \\
     & \leq & \{\sum_i [\sum_{j, k } I_{ijk}]^2\}^{\frac{1}{2}} \times [\sum_i 1^2]^{\frac{1}{2}}  \mbox{ (Cauchy-Schwarz inequality)} \vspace{0.3cm}\\
     & = & \{\sum_i[\sum_j (\sum_k I_{ijk}) \times 1]^2\}^{\frac{1}{2}} \times N^{\frac{1}{2}}  \vspace{0.3cm} \\
     & \leq & \{\sum_i[\sum_j(\sum_k I_{ijk})^2] \times N\} ^{\frac{1}{2}} \times N^{\frac{1}{2}}   \mbox { (Cauchy-Schwarz inequality)} \vspace{0.3cm}\\
     & = & N \{\sum_i[\sum_j(\sum_{k_1}I_{ijk_1})(\sum_{k_2}I_{ijk_2})]\}^{\frac{1}{2}} \vspace{0.3 cm} \\
     & = & N [\sum_{k_1 \neq k_2} (\sum_{i, j} I_{ijk_1} \times I_{ijk_2}) + \sum_{k_1 = k_2 = k} (\sum_{i, j} I_{ijk}^2)]^{\frac{1}{2}} \vspace{0.3cm} 
\end{array}
\]

First let's look at the second term in the bracket. Since $\sum_{i,j}I_{ijk}^2 \leq \sum_{i,j} 1 = N^2$ for each $k$, the second term is $\leq N^3$. 
Next let's look at the first term in the bracket. There are $N(N-1)$ pairs of $(k_1, k_2),$ where $k_1 \neq k_2$. For each such pair $(k_1, k_2)$, if 
we could find $i_1 \neq i_2$, $j_1 \neq j_2$ such that $I_{ijk_1} = I_{ijk_2} = 1$ for $i = i_1, i_2$ and $j = j_1, j_2$, then we get a rectangular box as shown in Figure~\ref{fig:1}.

\begin{figure}
\includegraphics[width = 5cm]{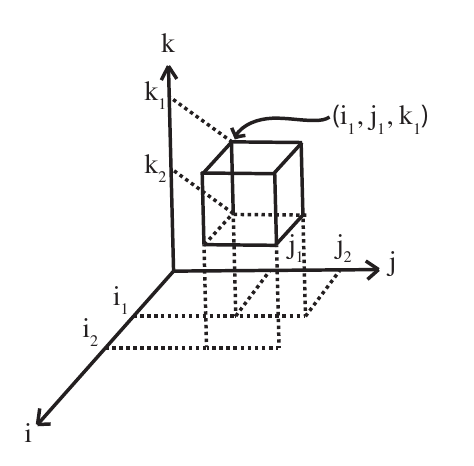}
\caption{When $I_{ijk} = 1$ for $i = i_1, i_2, j= j_1, j_2, k= k_1, k_2$, where $i_1 \neq i_2, j_1 \neq j_2, \text{ and } k_1 \neq k_2$, we get a rectangular box inside the grid}
\label{fig:1}
\end{figure}
This contradicts with our hypothesis that there is no rectangular box inside the grid whose eight vertices are contained in $A$. So for each $k_1 \neq k_2$, the maximum number of $(i, j)$ with $I_{ijk_1} = I_{ijk_2} =1$ is equal to the maximum number of vertices in an $N \times N$ grid which do not contain the four vertices of any rectangle. 

Let $B$ be a maximal set of vertices in an $N \times N$ grid that satisfies the condition, and let's define the following function:
\[ J : \{ (i, j) | 1 \leq i, j \leq N\} \longrightarrow \{0, 1\}, \] where
 \[
J(i, j) = \left\{ \begin{array}{rcl}
1 & \mbox{ if } & (i, j) \in B, \\
0 & \mbox{ if } & (i, j) \notin B. 
\end{array} \right.
\]Here each vertex is indexed by $(i, j)$, and we can use $J_{ij}$ as a short-hand notation for $J(i,j)$. 

Then we apply the Cauchy-Schwarz inequality again to obtain an upper bound for $B$. The calculation is shown as below:
\[
\begin{array}{rcl}
|B| & = & \sum_{i} \sum_{j} J_{ij}   \vspace{0.3cm} \\
     & = & \sum_i (\sum_{j} J_{ij}) \times 1  \vspace{0.3cm} \\
     & \leq & [\sum_i (\sum_{j } J_{ij})^2 ]^{\frac{1}{2}} \times [\sum_i 1^2]^{\frac{1}{2}}  \mbox{ (Cauchy-Schwarz inequality)} \vspace{0.3cm}\\
     & = & [\sum_i (\sum_{j_1}J_{ij_1})(\sum_{j_2}J_{ij_2})]^{\frac{1}{2}} \times N^{\frac{1}{2}} \vspace{0.3cm} \\
     & = & [\sum_i (\sum_{j_1 \neq j_2} J_{ij_1}I_{ij_2} + \sum_i \sum_{j = j_1 = j_2} J_{ij}^2)]^{\frac{1}{2}} \times N^{\frac{1}{2}} \vspace{0.3 cm} \\
     & = & [\sum_{j_1 \neq j_2} \sum_i J_{ij_1}J_{ij_2} + \sum_{j=j_1=j_2} \sum_i J_{ij}^2]^{\frac{1}{2}}\times N^{\frac{1}{2}} \vspace{0.3cm}
\end{array}
\]

Again let's first look at the second term in the bracket. Since $\sum_i J_{ij}^2 \leq N$ for each $j$, the second term is $\leq N^2$. Next let's analyze the first term. For each pair of $(j_1, j_2)$, with $j_1 \neq j_2$, we can have at most one $i$ such that $I_{ij_1} = I_{ij_2} =1$. Therefore, the first term is at most $N(N-1)$. Then we get the following estimate:
\[
|B|  \leq  [N(N-1) + N^2]^{\frac{1}{2}} \times N^{\frac{1}{2}} = [2N^2 - N]^{\frac{1}{2}} \times  N^{\frac{1}{2}} < \sqrt{2} N^{\frac{3}{2} }.
\]
Therefore $|B| = O(N^{\frac{3}{2}})$. 

Now we continue our estimate for the size of $A$ by using this result. Earlier we found that 
\[
|A| \leq N [\sum_{k_1 \neq k_2} (\sum_{i, j} I_{ijk_1} \times I_{ijk_2}) + \sum_{k_1 = k_2 = k} (\sum_{i, j} I_{ijk}^2)]^{\frac{1}{2}}, 
\]
and the second term is $\leq N^3$. Now we can bound the first term in the bracket by observing that for each pair $(k_1, k_2)$, with $k_1 \neq k_2$, the maximum number of $i, j$ such that $I_{ijk_1} = I_{ijk_2} = 1$ is $ <  \sqrt{2} N^{\frac{3}{2}}$, and so the first term is $ <  \sqrt{2} N^{\frac{3}{2} }\times N(N-1)$. Thus we conclude that $|A| < O(N^{\frac{11}{4}})$ as follows:
\[
|A| < N[\sqrt{2} N^{\frac{3}{2} }\times N(N-1) + N^3]^{\frac{1}{2}} < N[\sqrt{2}N^{\frac{7}{2}} +N^3]^{\frac{1}{2}}  < 2 N^{\frac{11}{4}}.
\]
\end{proof}

However, we don't know whether $N^{\frac{11}{4}}$ is a sharp upper bound of A. According to the paper {\it Remarks on the Box Problem} \cite{1}, Erd\"{o}s  conjectured it to be sharp and Katz and etc.\ \cite{2} have found an example that is $O(N^{11/4})$. So this paper attempts to approach this question from two points of view. First, we use the method of exhaustion to look at specific examples for small $N$, then we look at an equivalent two-dimensional version of the same question and hope to find an explicit formula.

% there is an estimate on the measure of the set and its connection with the discrete example, we can also use Cauchy-Schwartz to get an upper bound. Then each n corresponds to the n in the discrete case. 

\section{The discrete case} 
\subsection{Symmetries}
First let's look at the simplest example of $N=2$. When $N = 2$, we have a $2 \times 2 \times 2$ grid as shown in Figure~\ref{fig:2}. 

\begin{figure}[tb]
\includegraphics[width = 2cm]
{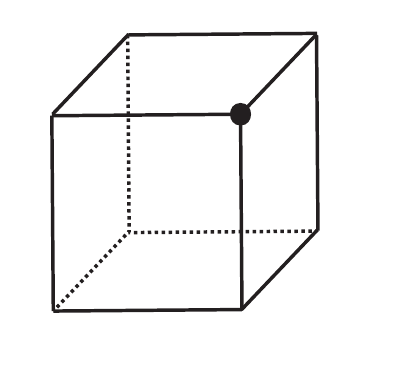}
\caption{We need to remove only one vertex from a $2 \times 2 \times 2 $ grid }
\label{fig:2}
\end{figure}

There is only one box inside the grid. so we need to remove just one vertex to get rid of the box. Therefore any set of seven vertices satisfies the condition that it does not contain the eight corners of any box inside the grid. 

\begin{defn}
A {\it configuration} is a maximal set of vertices in an $N \times N \times N$ grid such that it does not contain all the vertices of any rectangular box inside the grid. 
\end{defn}

\begin{defn}
The {\it order} of an $N \times N \times N$ grid is the order of a configuration for the $N \times N \times N$ grid. 
\end{defn}
As shown in our example above, the order of a $2 \times 2 \times 2$ grid is 7, and it has 8 different configurations. 

Although configurations are not necessarily unique, they can be obtained from one another through many symmetries. For example, given one configuration, we can rotate the grid and get another different configuration. In this way, we get many configurations that are essentially the same. In order to avoid this type of repetition, we study all possible symmetries that could arise, and hope that we could reach some kind of uniqueness in the end. To help us visualize the symmetries, a $3 \times 3 \times 3$ grid will serve as an example.

Let's begin with discussing the first type of symmetry. We observe that in a $3 \times 3 \times 3$ grid, there are nine two-dimensional $3 \times 3$ grids inside. More precisely, there are three two-dimensional $3 \times 3$ grids parallel to the $xy$-plane, three parallel to the $xz$-plane, and three parallel to the $yz$-plane. For example, a $3 \times 3$ grid parallel to the $xy$-plane is shown in Figure~\ref{fig:3}, which is highlighted in red. Furthermore, we give it a new name as follows. 

\begin{figure}
\label{fig:3}
\includegraphics[width = 5cm]{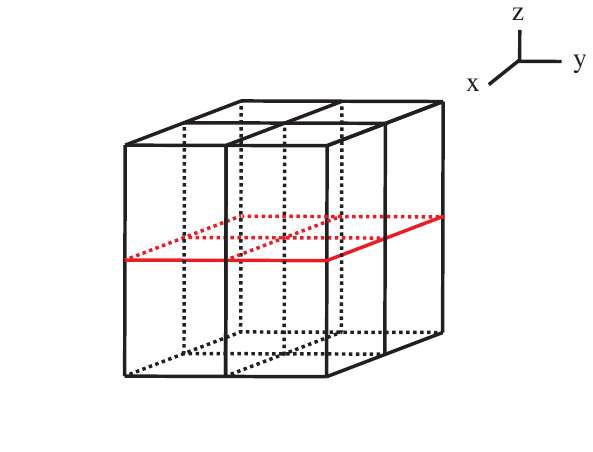}
\caption{An example of a layer in a $3 \times 3 \times 3$ grid.}
\end{figure}

\begin{defn}
A {\it layer} in an $N \times N \times N$ grid is a two-dimensional $N \times N$ grid. 
\end{defn}

There are $3N$ layers in an $N \times N \times N$ grid. It turns out that after we permute two different layers either from top to bottom, or from front to back, or from left to right, a configuration is sent to another configuration.  And this is our first type of symmetry. 

\begin{lem} 
\label{lem:2}
Given an arbitrary configuration, its image after any of the following three kinds of permutations is still a configuration (see Figure~\ref{fig:5}). 
\begin{enumerate}
\item permute any two layers from top to bottom;
\item permute any two layers from front to back;
\item permute any two layers from left to right.
\end{enumerate}
\end{lem}

\begin{figure}[tb]
\includegraphics[width = 10cm]
{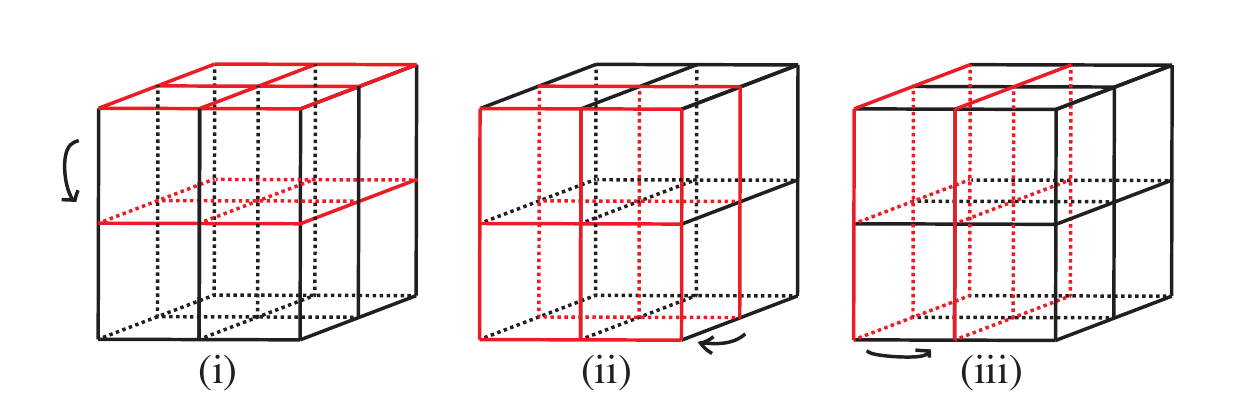}
\caption{When $N = 3$, we can permute any two layers (in red): (i) from top to bottom; (ii) from front to back; (iii) from left to right.}
\label{fig:5}
\end{figure}

\begin{proof} 
The proof is based on the observation that any of these permutations sends one rectangular box to another (possibly of a different size). Therefore if a set of vertices does not contain the eight corners of any box, then the new set of vertices after a permutation still does not contain all the corners of any box. 
\end{proof}

For $N = 2$, we see that all configurations can be obtained from one another by permuting the appropriate layers. Therefore we say the $2 \times 2 \times 2$ grid has a unique configuration (up to permutation). However, besides the permutations in Lemma~\ref{lem:2}, there are also other symmetries that could greatly simplify things. Now let's go to the second type of symmetry that we mentioned earlier. 

\begin{lem}
\label{lem:3}
When we view the grid as a cube, any of the 24 rotational symmetries of the cube sends a configuration to another. 
\end{lem}

\begin{proof} 
First, let's list all of the 24 rotational symmetries of a cube. They are the identity, the nine rotations whose axes pass through the centers of faces, the 6 rotations whose axes pass through the centers of edges, and the 8 rotations whose axes pass through the vertices. 
Second, we can check one by one that each rotation sends a rectangular box to another. 
Third, since a rotation is invertible, if a set of vertices does not contain all the vertices of any box inside the grid, after applying a rotation, the new set of vertices still doesn't contain any box's eight vertices. 
Thus a configuration is sent to another. 
\end{proof}

The third type of symmetry is the reflectional symmetry. 
\begin{lem}
\label{lem:4}
There are 10 reflectional symmetries of a cube that sends a configuration to another. 
\end{lem}

\begin{proof}
First, we consider reflections of an $N \times N$ grid. There are four reflections in total: two with respect to the diagonals, and two with respect to the middle lines. Each reflection sends a rectangle to another as grid lines are sent to grid lines.
Next, we look at the $N \times N \times N$ grid from the top, we see an $N \times N$ grid. It follows that the four reflections of the $N \times N$ grid give rise to four reflectional symmetries of the $N \times N \times N$ grid. This is because boxes are mapped to boxes. 
Then, we can also look from the front or the right. In total these account for the 9 nontrivial reflectional symmetries of the $N \times N \times N$ grid. Adding the identity map gives us the 10 reflections in total. Therefore all the reflectional symmetries send rectangular boxes to rectangular boxes. 
Last, since a reflection is invertible, it sends a configuration to another. 
\end{proof}

Let $G$ be the group generated by the three types of symmetries shown above, and we wonder whether $G$ includes all possible `maps' that could send one configuration to another. First of all, we need to make sense of the `maps' here. Suppose $f: \{\text{vertices}\} \rightarrow \{\text{vertices}\}$ is any map that sends one configuration to another, $f$ needs to send a rectangular box to another rectangular box. Therefore, we want $f$ to be a bijective map between the vertices such that if eight vertices determine a box in the grid, then their images under the map $f$ also determine a box. Let $H$ be the set of all such maps, then $H$ is a group under the operation of composition. Then is $H$ equal to $G$? It's obvious that $G$ is contained in $H$, so let's check whether the reverse relation is also true.

\begin{thm}
Given an $N \times N \times N$ grid, let $G$ be the group generated by three types of symmetries, namely permutations, rotations, and reflections. Let $H$ be the group of bijective maps $f$ between the vertices of the grid such that if eight vertices determine a box in the grid, then their images under $f$ also determine a box. When $N=2$, $H$ is strictly larger than $G$. Moreover, when $N \geq 3$, $H = G$. 
\end{thm}

\begin{proof}
When $N = 2$, there is only one box inside the $2 \times 2 \times 2$ grid, thus $H$ is the permutational group $S_8$ of all the vertices. Since $S_8$ is generated by all the transpositions, it suffices to check that $(12)$ is in $G$. It turns out that $(12)$ is not in $G$, because every map in $G$ sends adjacent vertices to adjacent vertices. (However, we've seen that $G$ contains enough maps to send the eight possible configurations of a $2 \times 2 \times 2$ grid into one another.) 

When $N \geq 3$, we'll show that $H = G$. 
First, given two adjacent vertices in an $N \times N \times N$ grid, we check whether their images under the map $f$ are also adjacent to each other. This is true. After permuting appropriate layers, their images can be made into a $1 \times 1 \times 1$ box inside the grid. We observe that there are three possible positions for them, namely they are adjacent to each other, they are on the diagonal of a face, or they are on the diagonal of the box (see Figure~\ref{fig:21}). 

\begin{figure}[tb]
\includegraphics[width = 4cm]
{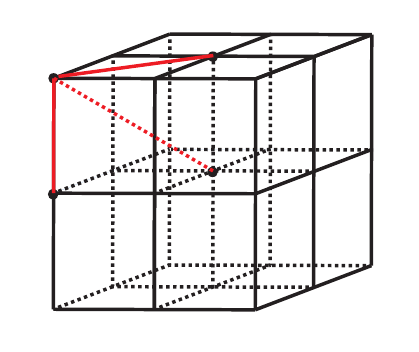}
\caption{When $N = 3$, two vertices in the grid can be made into one of these three relative positions (in red) after permuting appropriate layers and some rotations.}
\label{fig:21}
\end{figure}

If two vertices are adjacent, there are $(N-1)^2$ rectangular boxes in the grid containing them. Suppose their images under the map $f$ are not adjacent. If they are on the diagonal of a face, there are $(N - 1)$ rectangular boxes containing them. If they are on the diagonal of the $1 \times 1 \times 1$ box, there is only one box containing them. Since $f$ sends a box to a box, it follows that $(N-1)^2$ boxes are sent to either $(N-1)$ boxes or one box. This is a contradiction to the hypothesis that $f$ is bijective between vertices. Therefore, adjacent vertices are sent to adjacent vertices. Moreover, we can deduce that $f$ sends an edge of a box to an edge of another box. 

Second, let's prove that all vertices in the same layer are mapped to vertices in the same layer. It suffices to argue for the case when $N = 3$. Consider the 9 labeled vertices shown in Figure~\ref{fig:22}, it is not hard to see $f(1)$, $f(2)$, $f(3)$, and $f(4)$ are vertices of a face in some layer. Since there is an edge between 1 and 3, an edge between 3 and 5, and an edge between 3 and 5, it follows that 5 has to be in the same layer as $f(1)$, $f(2)$, $f(3)$, and $f(4)$. Similarly, 6 is also in the same layer as $f(1)$, $f(2)$, $f(3)$, and $f(4)$. Repeat the same argument for 7, 8, and 9, so the vertices from 1 to 9 are sent to vertices in the same layer. 

\begin{figure}[tb]
\includegraphics[width = 4cm]
{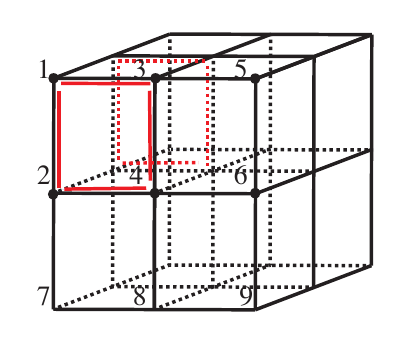}
\caption{Vertices from 1 to 9 in the front layer are sent under the map $f$ to vertices in the same layer.}
\label{fig:22}
\end{figure}

Third, given two parallel layers, let's shown that their images under the map $f$ are parallel layers. Look at the two opposite faces highlighted in red in Figure~\ref{fig:22}, using the fact that $f$ sends an edge to an edge, it's not hard to see that they are sent to two opposite faces in another box. Therefore, the two layers containing these two faces are sent to another two parallel layers. 
 
\begin{defn}
Let's call the layers that are parallel to the $yz$-plane {\it X-layers}. Similarly, we call the layers that are parallel to the $xz$- and $xy$-planes {\it Y-layers} and {\it Z-layers}, respectively.
\end{defn} 

Fourth, since $f$ permutes the three sets of $\{$X-layers$\}$, $\{$Y-layers$\}$, and $\{$Z-layers$\}$, there exists a reflectional or rotational symmetry $\pi_1$ of the grid such that $\pi_1 \circ f$ preserves each of $\{$X-layers$\}$, $\{$Y-layers$\}$, and $\{$Z-layers$\}$. Furthermore, we can permute the X-layers so that each X-layer is preserved; that is to say, there is a permutation $\pi_2$ such that $\pi_2 \circ \pi_1 \circ f$ preserves each X-layer. Similarly, there exists $\pi_3$ and $\pi_4$ such that $\pi_4 \circ \pi_3 \circ \pi_2 \circ \pi_1 \circ f$ preserves each Y-layer and each Z-layer. Let's denote the map $\pi_4 \circ \pi_3 \circ \pi_2 \circ \pi_1 \circ f$ as $f'$. Pick one X-layer and it is an $N \times N$ grid. The vertices under the map $f'$ satisfy the property that if four vertices determine a rectangle in the $N \times N$ grid, then their images also determine a rectangle. 

We claim that for an $N \times N$ grid, if a map $g$ is a bijective map between its vertices and $g$ sends a rectangle's four vertices to another rectangle's four vertices, then $g$ is in the group generated by permutations of layers, rotational symmetries, and reflectional symmetries of the $N \times N$ grid, which are defined similarly to the $N \times N \times N$ grid. For example, a layer in the $N \times N$ grid is a one-dimensional $N$ grid. The proof of the claim follows the same steps as we are showing for the $N \times N \times N$ grid. 

A permutation of layers in the $N \times N$ grid can be realized by a permutation of Z- or Y-layers in the $N \times N \times N$ grid; a rotation of the $N \times N$ grid can be realized by a rotation of the $N \times N \times N$ grid; and a reflection of the $N \times N$ grid can also be realized by a reflection of the $N \times N \times N$ grid. Therefore, there exists $\pi_5 \in G$ such that $\pi_5 \circ f'$ preserves the vertices of a X-layer. Since two opposite faces of a box in the $N \times N \times N$ grid are sent to two opposite faces in another box, it implies that $\pi_5 \circ f'$ is the identity map. Thus $f = \pi_1^{-1} \circ \ldots \circ \pi_5^{-1}$ is a map in $G$, and so $H = G$. 
\end{proof}

\begin{defn}
We say that two configurations are {\it equivalent} to each other if one can be obtained from the other through one of the symmetries in $G$.
\end{defn} 

\subsection{The discrete case: $N=3$}
Here we describe one strategy of looking for a configuration in a grid, that is to search for the minimal set of vertices that we want to {\it remove}. There are fewer points to remove than to add. Using this strategy, we can study our first meaningful case when $N$ = 3. 

To save a lot of writing, when we say we {\it mark} a point in the grid, we really mean to mark one {\it removed} point in the grid. And a box is {\it marked} if it has a removed point, otherwise we call it {\it unmarked}. So in order to satisfy the hypothesis, we want every box inside the grid to be marked. What is the minimum number of (removed) points we need to mark? To answer this question, we use the method of exhaustion when $N = 3$.

\begin{thm}
\label{thm:1}
The order of a $3 \times 3 \times 3$ grid is 22. Equivalently, the minimum number of points we need to remove from a $3 \times 3 \times 3$ grid so that the remaining set of vertices does not contain all eight vertices of any rectangular box inside is 5. Moreover, a configuration for the $3 \times 3 \times 3$ grid is unique up to equivalence. 
\end{thm}

\begin{proof}
First, we can mark the first point as shown in Figure~\ref{fig:7}(i), after permuting the appropriate layers. It's not hard to see that there are still unmarked boxes inside. So we need to mark a second point. 

\begin{figure}[tb]
\includegraphics[width = 16cm]
{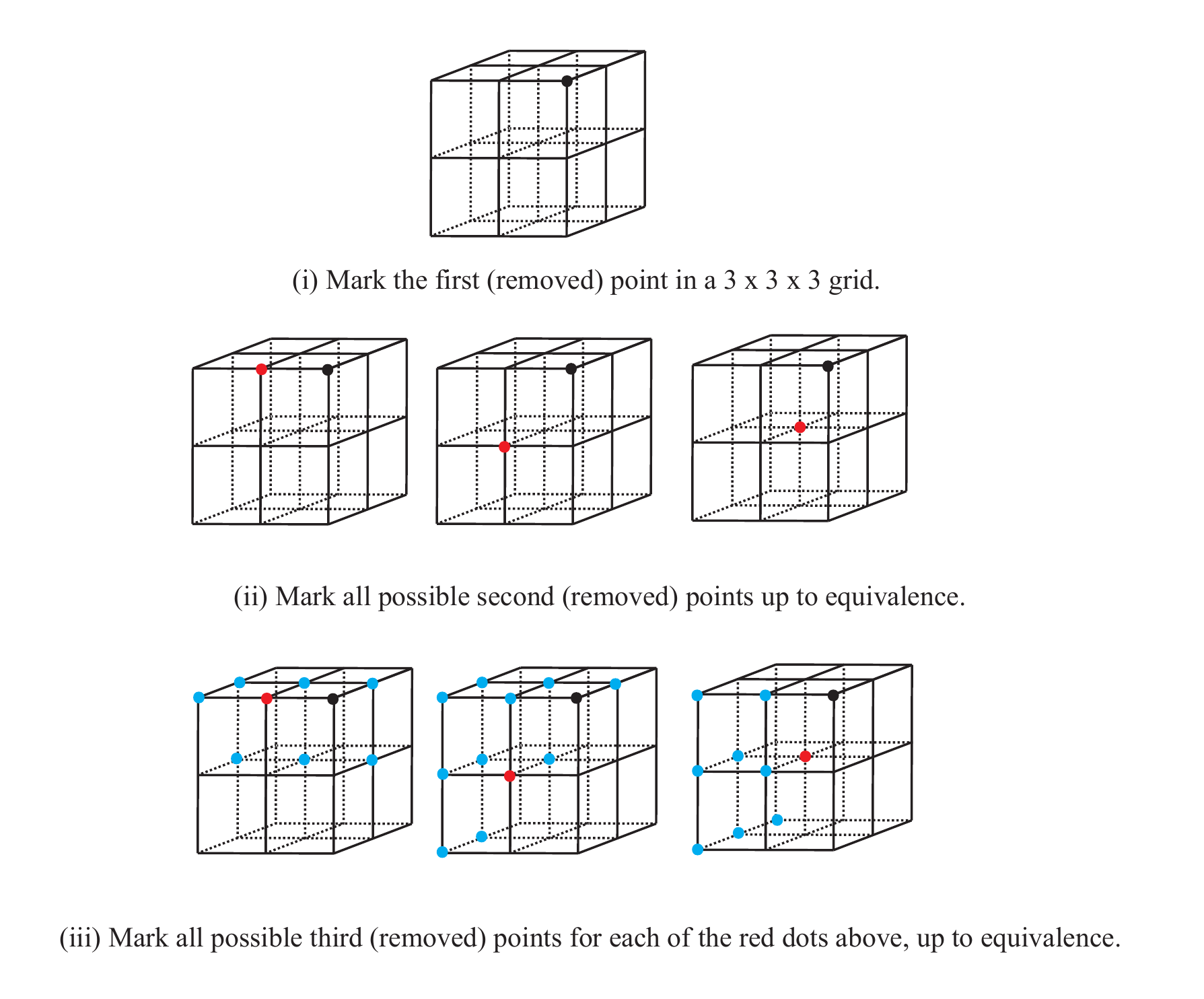}
\caption{}
\label{fig:7}
\end{figure}

Next, up to equivalence, there are three possible choices for the second point as shown in Figure~\ref{fig:7}(ii) as red dots. This is because through permutational symmetry the second point can be transported to the unit cube containing the first dot. Then we can use rotational symmetry to eliminate 4 of the 7 possible vertices. Therefore it gives us the three cases in Figure~\ref{fig:7}(ii).

It's not hard to see that we need to mark a third point in each of these three cases. For each case, we use blue dots to indicate all possible third points that we want to mark up to equivalence. This is shown in Figure~\ref{fig:7}(iii). 

Since there are quite a few blue dots, we introduce another useful tool in order to reduce the amount of work in the method of exhaustion. Before stating it formally, we will first look at an example to help us better understand the idea. 

Let's look at the rightmost grid in Figure~\ref{fig:7}(iii). If we look at it from the right-hand side, we get a square with some dots as illustrated in Figure~\ref{fig:10}. Notice that at the vertex in the center, there are two dots: one is blue and the other is red. This is because there is an overlapping if we look at the grid from the right-hand side. 

\begin{figure}[tb]
\includegraphics[width = 4cm]
{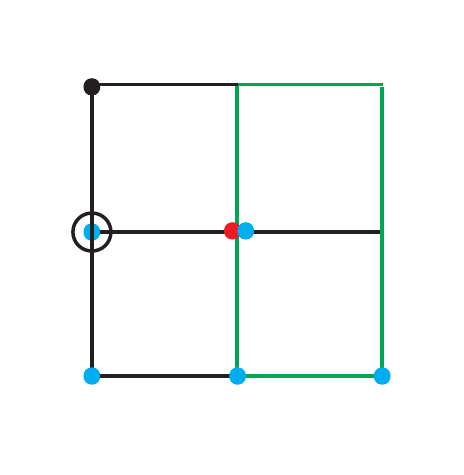}
\caption{A side view of the rightmost cube in Fig~\ref{fig:7}(ii); for the blue dot in circle, the green rectangle is unmarked.}
\label{fig:10}
\end{figure}

\begin{lem}
\label{lem:5}
For each of these blue dots in Figure~\ref{fig:10}, if a rectangle has only one dot, then we need to mark a fourth point to the grid; moreover, if there is unmarked rectangle, we need to mark fourth and fifth points to the grid. 
\end{lem}

\begin{proof}
Each rectangle in this $2 \times 2$ grid corresponds to three distinct rectangular boxes, and at least two dots are needed to mark all three of them. Therefore, if a rectangle has only one dot, we need to mark another; if a rectangle has none, we need to mark two more.  
\end{proof}

Let's apply this lemma to the blue dot in circle in Figure~\ref{fig:10} as an illustration. In the figure, we see that there is a rectangle that has not been marked yet and it is shown in green. Therefore, we need to mark two more points in the original grid given that the third marked point is this particular blue dot. 

Similarly, we can check for the other blue dots. It turns out that five of the six blue dots need two more points, except for the one on the diagonal as indicated in Figure~\ref{fig:12}. Before proceeding with this particular dot, we check the other two grids in Figure~\ref{fig:7}(iii) and find that all their blue dots also need two more points. So we are left with only the blue dot in Figure~\ref{fig:12} to check! 

\begin{figure}[tb]
\includegraphics[width = 12cm]
{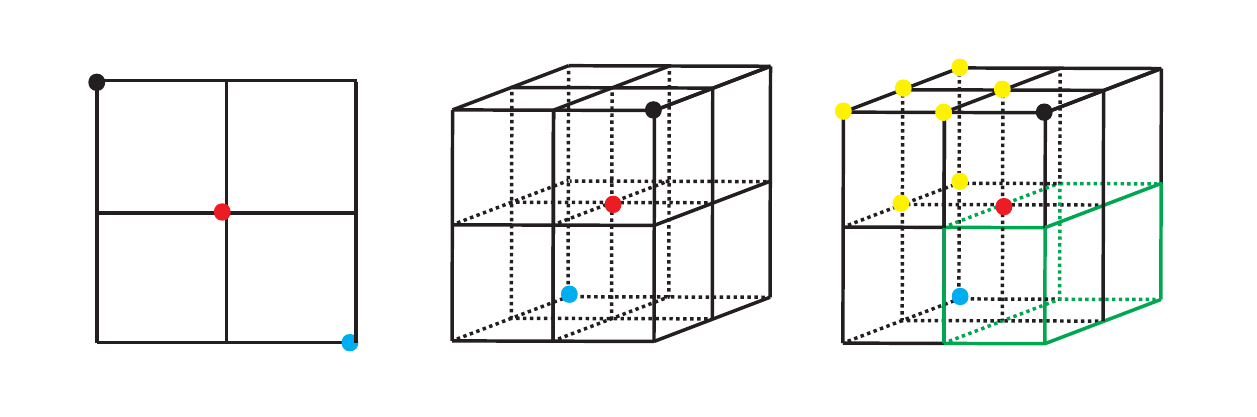}
\caption{The last blue dot that needs to be checked in Figure~\ref{fig:10}; so mark all possible fourth (removed) points to the grid.}
\label{fig:12}
\end{figure}

We need to mark a fourth point. There are two rotational symmetries and three reflectional symmetries fixing the three dots in Figure~\ref{fig:12}. This is from the observation that every such symmetry needs to fix the diagonal on which the black, red, and blue dots lie. As a result, applying these symmetries, we only need to consider 7 possible candidates for the fourth point. This is shown in Figure~\ref{fig:12} as yellow dots. It is easy to see that there is an unmarked rectangular box for all of the yellow dots (shown in green in the same figure). Therefore we need to mark a fifth point.

So far we've shown that we need to mark at least five points in an attempt to mark every rectangular box in the $3 \times 3 \times 3$ grid. Now let us check whether a sixth point is needed. We will see that five points are sufficient to get a configuration and there is a unique way of arranging these five points up to equivalence. Before proceeding, we need to introduce a definition.

\begin{defn}
Given any finite collection of parallel layers, we can project them in the orthogonal direction and get an $N \times N$ grid with some (possibly overlapping) dots, which we calle a {\it side diagram}. 
\end{defn}

As an illustration, look at the leftmost $3 \times 3 \times 3$ grid in Figure~\ref{fig:7}(iii). We pick the parallel layers that have the red and black dots and project them to the right. Then we obtain a side diagram as shown in Figure~\ref{fig:14}. Note that since each blue dot represents a possible third point, we don't have an overlapping of the blue dots in the side diagram but we do for the red and black dots. The following lemma gives a convenient way for checking whether a grid has any unmarked rectangular box. 

\begin{figure}[tb]
\includegraphics[width = 8cm]
{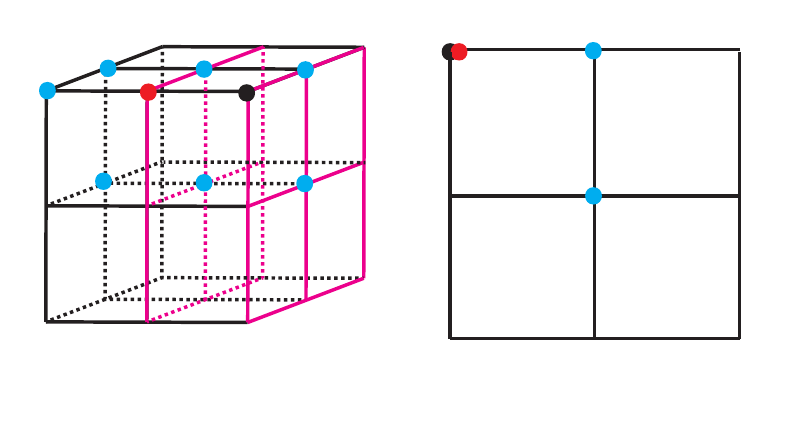}
\caption{Given two layers in pink, their side diagram if we see from the right.}
\label{fig:14}
\end{figure}

\begin{lem}
\label{lem:6}
An $N \times N \times N$ grid satisfies the condition that every rectangular box inside it is marked if and only if after we pick any two parallel layers, their side diagram has no unmarked rectangle. 
This is also equivalent to if we look at the grid just from the right (or the front, or the top) and pick any two layers, then their side diagram has no unmarked rectangle.
\end{lem}

Let's continue our proof of the Theorem~\ref{thm:1}. 

1.  For the leftmost grid in Figure~\ref{fig:7}(ii), we look at it from the front and show that a sixth point is needed, so that a five-point configuration can't be found in this case. If we look at the back two layers from the front (i.e.\ the two layers that don't contain the red and black dots), we need at least three points in them, because a $3 \times 3$ grid needs at least 3 points in order to mark all the rectangles. Once we mark 3 points, one of the two layers has no more than 2 dots. Pick this layer and the front one, their side diagram always has an unmarked rectangle, because the front layer needs at least 4 dots in order to mark all rectangles (see Figure~\ref{fig:15}). 

2. For the middle grid in Figure~\ref{fig:7}(ii), we look at it from the right and show that there exists only one five-point configuration (not up to equivalence). In the side diagram of the first two layers from the right, we need to have at least two extra dots in order to mark all the rectangles as shown in Figure~\ref{fig:15}. There are only three different ways to put these two dots if we take symmetries into account.

\begin{figure}[tb]
\includegraphics[width = 10cm]
{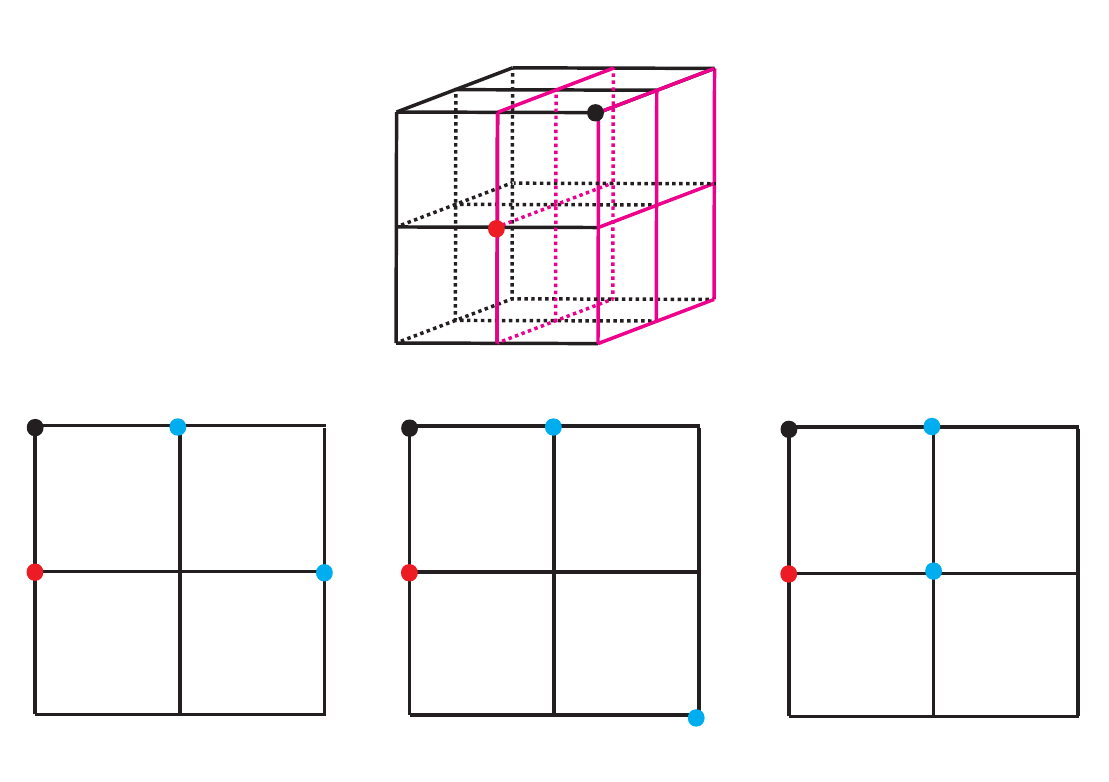}
\caption{In the side diagram of the two layers in pink , at least two extra dots (in blue) are needed in order to mark all rectangles.}
\label{fig:15}
\end{figure}

If the two extra dots are in the same layer, we must mark at least another two dots. For example, if the two extra dots are in the first layer from the right, then we pick the second and the third layers, and obtain a side diagram with only one dot. It implies that we must add at least another two dots. Therefore, six points are needed at minimum in this situation.

If each of the first two layers contains only one dot, we break into the following three cases. 

{\it Case I:} The leftmost side diagram in Figure~\ref{fig:15} needs two more points, because we can check the two possibilities as shown in Figure~\ref{fig:16}. For the left grid in Figure~\ref{fig:16}, the first and third layers give rise to a side diagram that needs at least two more points. For the right one in Figure~\ref{fig:16}, the side diagram corresponding to the first and third layers needs only one dot. However, the side diagram coming from the second and third layers also needs one dot. Those two dots can't replace one another. Therefore we need to add at least two extra points in this case too. 

\begin{figure}[tb]
\includegraphics[width = 8cm]
{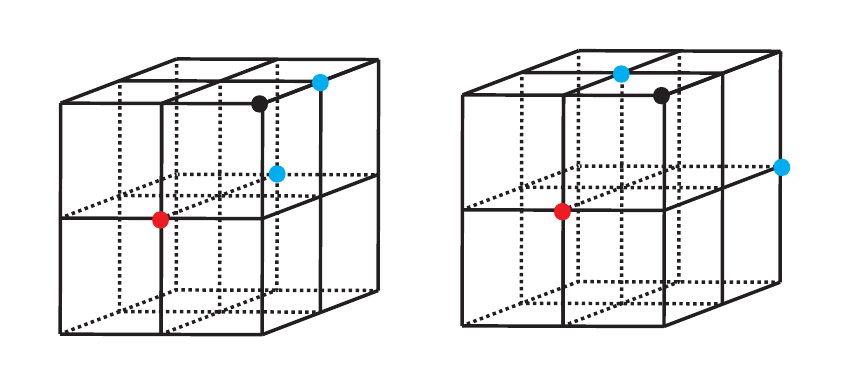}
\caption{Two possibilities for the leftmost side diagram in Figure~\ref{fig:15} if each layer in pink contains one blue dot.}
\label{fig:16}
\end{figure}

{\it Case II:} The middle side diagram in Figure~\ref{fig:15} also needs two more points, and the proof is similar to that of the previous case. 

{\it Case III:} The rightmost side diagram in Figure~\ref{fig:15} needs only one point. There are also two possibilities as in case I, but there is only one way to add one point so that all side diagrams are marked. This is shown in Figure~\ref{fig:17}. The corresponding grid is also shown in the same figure. So this is the only one five-point configuration and it is not up to equivalence. 

\begin{figure}[tb]
\includegraphics[width = 8cm]
{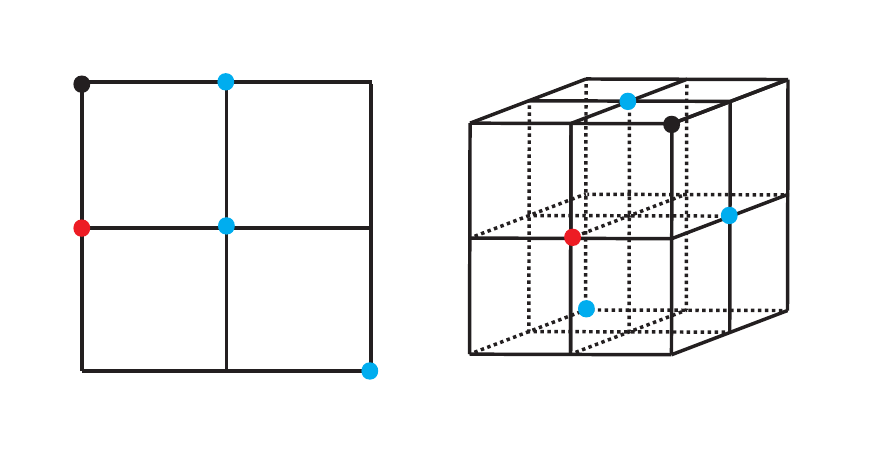}
\caption{The only five-point configuration for the rightmost side diagram in Figure \ref{fig:15} if each layer in pink contains one blue dot. }
\label{fig:17}
\end{figure}

3. For the rightmost grid in Figure~\ref{fig:7}(ii), we look at it from the right and show that there exist many possible five-point configurations. However, they are all equivalent to each other, and are also equivalent to the one we found in Figure~\ref{fig:17}. 

When looking at its all possible third-point positions as shown as the rightmost grid in Figure~\ref{fig:7}(iii), we obtain a side diagram as illustrated in Figure~\ref{fig:18}. There are three cases to consider as follows.

\begin{figure}[tb]
\includegraphics[width = 3.5cm]
{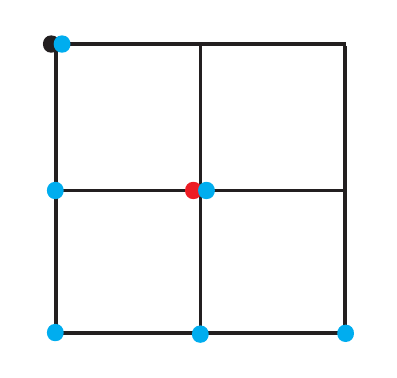}
\caption{The possible positions of a third point when looking at the rightmost grid in Figure~\ref{fig:7}(iii) from the right.}
\label{fig:18}
\end{figure}

{\it Case I:} When the blue dot overlaps with the black or red one. This brings us back to the leftmost grid in Figure~\ref{fig:7}(ii) after possible permutations of layers. We've shown that at least six points are needed for this grid. 

{\it Case II:} When the blue dot is right below the black or red one. Depending on whether the blue dot is in the same layer as the black or the red one, we can bring it back to either the leftmost or the middle grid in Figure~\ref{fig:7}(ii). So there could be many possible five-point configurations, but they are all equivalent to the one shown in Figure~\ref{fig:17}. 

{\it Case III:} When the blue dot is on the diagonal. If it is in the same layer as the black or the red dot, after permuting the appropriate layers, we go back to the middle grid in Figure~\ref{fig:7}(ii) again. If the blue dot is neither in the first layer nor the second one,  we need to look at all the possible positions of a fourth point. Up to reflectional symmetry, this is shown as yellow dots in Figure~\ref{fig:19}. After checking each yellow dot, we see that we can go back to either the rightmost or the middle grid in Figure~\ref{fig:7}(ii) again. So this completes our proof. 

\begin{figure}[tb]
\includegraphics[width = 3.5cm]
{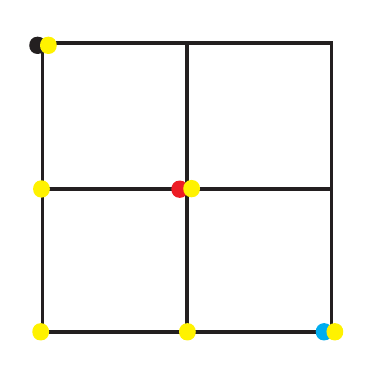}
\caption{The possible positions of a fourth point if the third point is on the diagonal in Figure~\ref{fig:18}.}
\label{fig:19}
\end{figure}

\end{proof}

\subsection{The discrete case: $N = 4$}
A $4 \times 4 \times 4$ grid is more complicated than a $3 \times 3 \times 3$ grid, because it has 64 points in it. The method of exhaustion as demonstrated above might be too much for a $4 \times 4 \times 4$ grid, but we will introduce a new technique to show that it is not hard to find the order in this particular case. 

\begin{thm}
\label{thm:2}
The order of a $4 \times 4 \times 4$ grid is 47. Equivalently, the minimum number of points we need to remove from a $4 \times 4 \times 4$ grid so that the remaining set of vertices does not contain all eight vertices of any rectangular box inside the grid is 17. Moreover, a configuration for the $4 \times 4 \times 4$ grid is not unique up to equivalence, and there are at least two equivalent classes. 
\end{thm}

\begin{proof}
First, let's show that the order of a $4 \times 4$ grid is 9. That is to say, the minimum number of points we need to remove from a $4 \times 4$ grid so that the remaining set of vertices does not contain all four vertices of any rectangle inside the grid is 7. Let's look at the four horizontal layers inside the $4 \times 4$ grid as shown in Figure~\ref{fig:23}. Similar to Lemma \ref{lem:6}, a $4 \times 4$ grid satisfies the condition that every rectangle inside is marked if and only if for any two horizontal layers, their side diagram in the $y$-direction has at least 3 marked (removed) points. First, we show that 7 is the minimum number of points one need in order to mark every rectangle. If we mark at least 2 points in every horizontal layer, then we mark at least 8 points in the $4 \times 4$ grid. On the other hand, if we mark only one point in one horizontal layer, then we need to mark at least 2 points for the other three horizontal layers, which require only 7 points in total at minimum. It turns out that 7 points are sufficient to mark all rectangles as shown in Figure~\ref{fig:23}, and it is not hard to see that the configuration is unique up to equivalence (start with the layer that has only one marked point). 

\begin{figure}[tb]
\includegraphics[width = 4cm]
{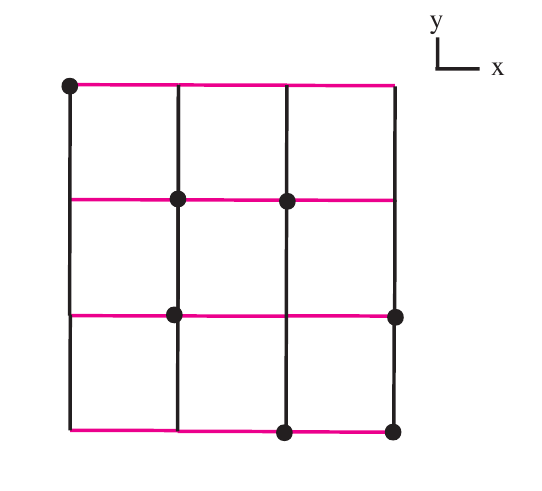}
\caption{A $4 \times 4$ grid needs at least 7 marked (removed) points such that the remaining set of vertices does not contain all four vertices of any rectangle inside the grid.}
\label{fig:23}
\end{figure}

Second, let's prove that the order of a $4 \times 4 \times 4$ grid is 47. If we pick any two Z-layers (layers parallel to the $xy$-plane), their side diagram must have at least 7 marked points. Therefore, we can get a lower bound on the number of marked points by listing all possible combinations as shown in the following table.

\begin{table} [ht]
\renewcommand\arraystretch{1.5}
\[
\begin{array}{|c|c|c|c|c|} 
\hline
\text{1st} &\text{2nd}&\text{3rd} &\text{4th} &\text{The total number}\\ \hline
0 & 7 & 7 & 7 & 21 \\ \hline
1 & 6 & 6 & 6 & 19 \\ \hline
2 & 5 & 5 & 5 & 17 \\ \hline
3 & 4 & 4 & 4 & 19\\ \hline
\end{array}
\]
\vspace{0.5cm}
\caption{All possible combinations of marked points in the four Z-layers of a $4 \times 4 \times 4$ grid.}
\end{table}
We find that at least 17 points need to be marked. If we could find a configuration with only 17 marked points, then the order of a $4 \times 4 \times 4$ grid is 47. One such configuration is shown in Figure~\ref{fig:24}, where the 3rd and 4th layers are obtained from the 2nd one by permuting the dots on the pink vertical bars. Let's check this is indeed a configuration. From construction, the side diagrams of (1) and (2), (1) and (3), and (1) and (4) are equivalent to the one shown in Figure~\ref{fig:23}, and thus have no unmarked rectangles. The side diagrams of (2) and (3), (2) and (4), and (3) and (4) are shown in Figure~\ref{fig:25}. It is easy to see that there are no unmarked rectangles in these side diagrams as well. 

\begin{figure}[tb]
\includegraphics[width = 12cm]
{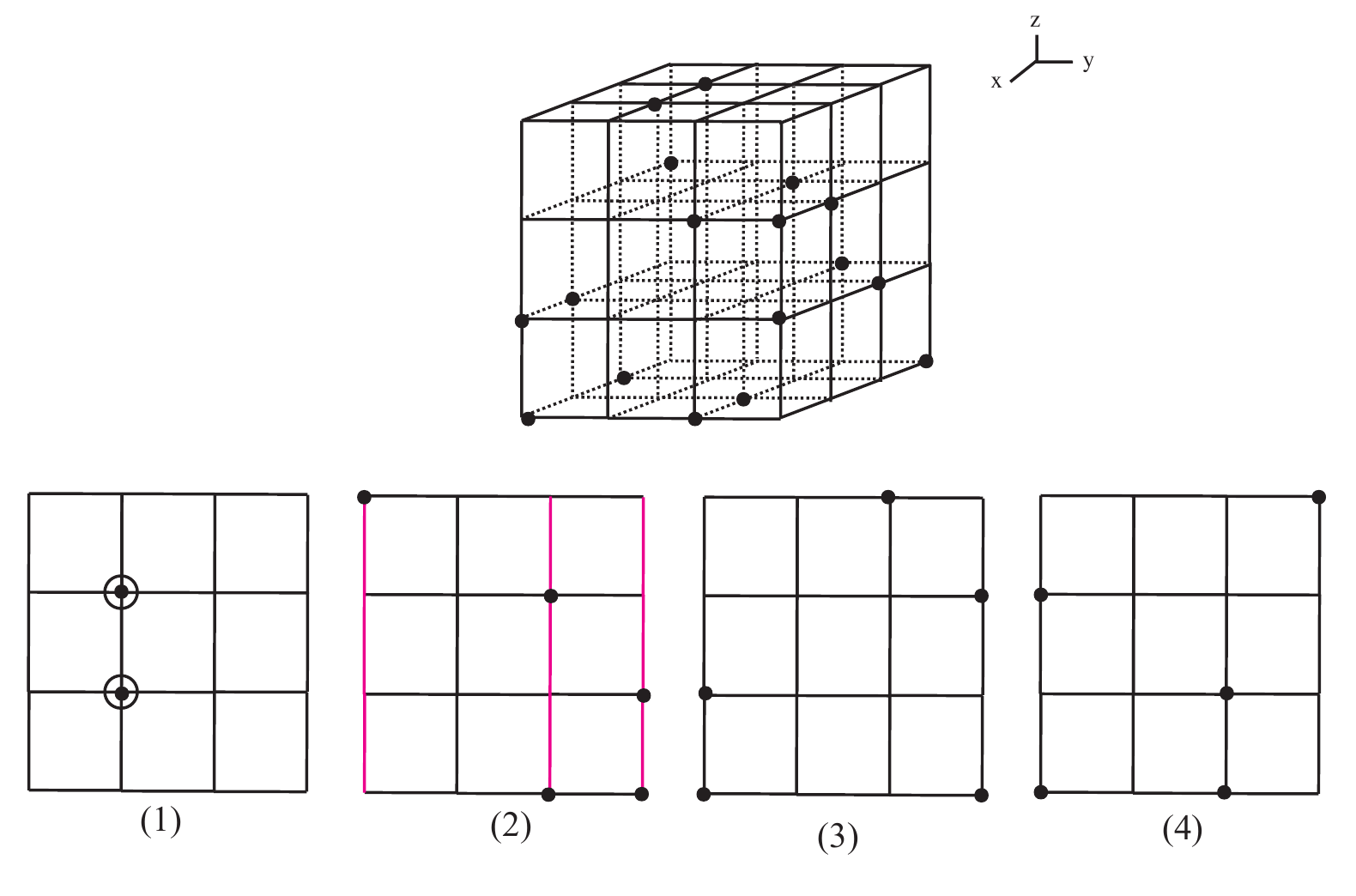}
\caption{One configuration of a $4 \times 4 \times 4$ grid with all Z-layers shown side by side.}
\label{fig:24}
\end{figure}

\begin{figure}[tb]
\includegraphics[width = 10cm]
{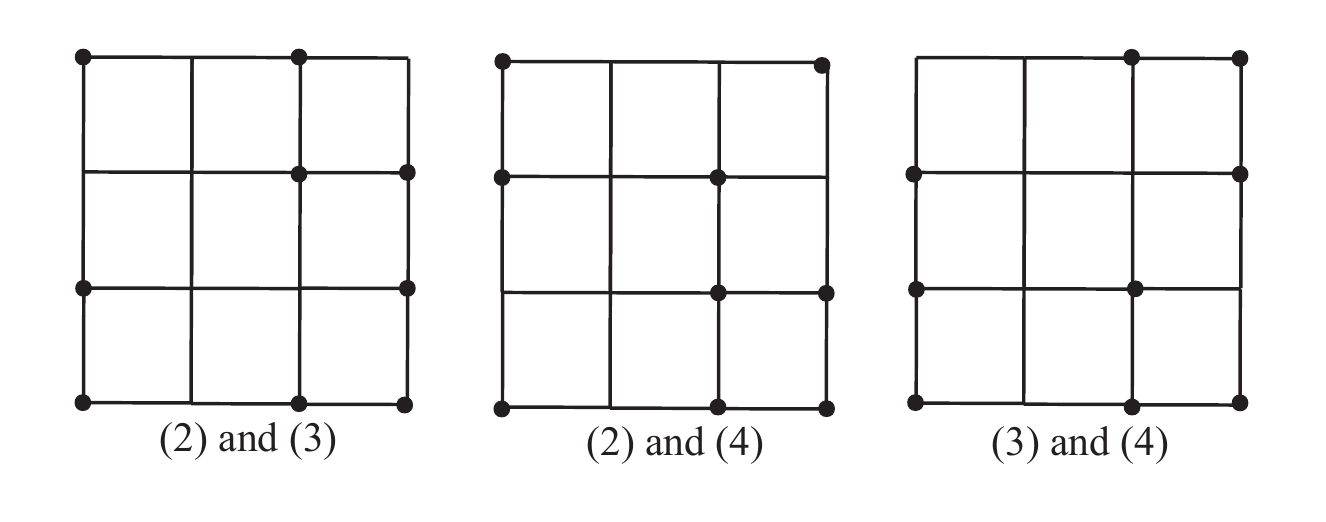}
\caption{The side diagrams of (2) and (3), (2) and (4), and (3) and (4) from the configuration in Figure~\ref{fig:24}.}
\label{fig:25}
\end{figure}

However, the configuration in Figure~\ref{fig:24} is not unique up to equivalence. We can come up with another configuration as shown in Figure~\ref{fig:26}. The reason why they are not equivalent is that the two dots (circled in two figures) in the first layer can be either on an edge or on a diagonal. If the two dots of a configuration are on an edge, we can show that this configuration is equivalent to the one in Figure~\ref{fig:24}. This is because through possible permutations and symmetries, we can make the first and second layers look like the ones shown in Figure~\ref{fig:24}. Then the third and fourth layers are determined. Similarly, if the two dots of a configuration are on a diagonal, then this configuration is equivalent to the one in Figure~\ref{fig:26}. 

\begin{figure}[tb]
\includegraphics[width = 12cm]
{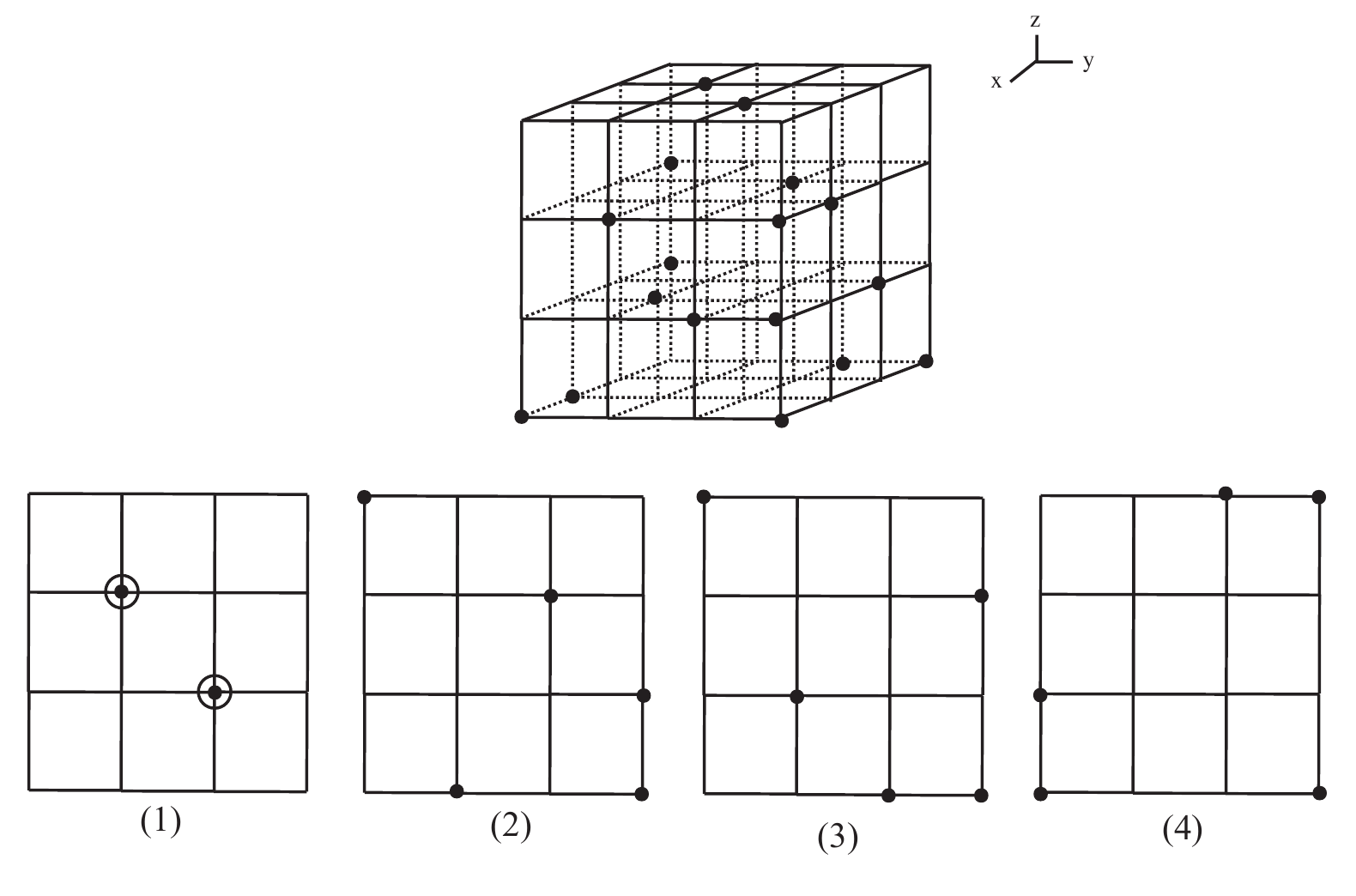}
\caption{Another configuration of a $4 \times 4 \times 4$ grid with all Z-layers shown side by side.}
\label{fig:26}
\end{figure}

\end{proof}

When $N$ gets larger, it becomes more difficult to find the the order of an $N \times n \times N$ grid. So instead of moving onto a $5 \times 5 \times 5$ grid, we hope to find another way to look at this problem. One way is to study the two-dimensional grids thoroughly and see if there is a pattern that can be generalized to the three-dimensional grids. So the next section devotes to learning two-dimensional $N \times N$ grids. There will be a notational change. We reserve the capital letter $N$ for lattice cubes and we will use the lower case $n$ for matrices later.

\section{Two dimensional $\bf N \times N$ grids}
Since an $N \times N$ grid is also an $n \times n$ matrix (here $N$=$n$), we can rephrase our question in the following way: what is the minimum number of 1's that we can put in an $n \times n$ matrix, whose entries are either 0 or 1, such that every $2 \times 2$ minor contains at least one 1? Here each 0 represents a selected vertex and each 1 represents a vertex that is not selected. If every $2 \times 2$ grid has at least one 1, no rectangle's four corners are selected simultaneously. So the original question concerning lattices has now been set in the new linear algebra frame. Our definitions of a configuration, an order, etc.\ are still applicable here, but we will limit their usage to avoid any confusion. 

Moreover, we ask a more general question: If $1 \leq k \leq n$, given an $n \times n$ matrix $A$ with entries 0 or 1, what is the minimum number of 1's such that any $k \times k$ minor contains at least one 1, or equivalently any $k \times k$ minor is not a zero matrix? We hope to find a pattern between different choices of $k$ so as to solve the question for $k = 2$. 

We begin with some easy observations. When $k=n$, there is only one $n \times n$ matrix, therefore the minimum number of 1's in $A$ is just one. On the other hand, when $k=1$, every entry in $A$ is a $1\times 1$ minor, so $n^2$ is the minimum number. Then we ask the question that as $k$ decreases from $n$ to 1, how does the minimum number gradually increases from 1 to $n^2$? Does it follow a specific pattern? 

Let's look at the next slightly harder example. When $k = n-1$, we can prove that 3 is the minimum number. This is because if there is only one 1, we may assume that it is in the upper left corner after permuting appropriate rows and columns. Then the $(n-1) \times (n-1)$ minor in the lower right corner of $A$ is a zero matrix as shown in bold below: 
\[  A = \left[ \begin{array}{cc}
				1	&       		\\
				        &      \left [\begin{array}{ccc}
				        					      \bf  0 &\ldots &\bf 0 \\
									        \vdots &  &\vdots \\
									       \bf  0&\ldots   &\bf 0
				        						\end{array} \right]    		
				
			\end{array}
		\right],\]				
which does not satisfy the hypothesis. Suppose there are only two 1's in $A$, again we may arrange the rows and columns so that one of the 1's is in the $(1,1)$-entry. If the second 1 is in the first column or the first row, the preceding argument gives a contradiction. Otherwise, we may assume that it is in the $(2,2)$-entry of $A$. Consider the $(n-1) \times (n-1)$ minor whose rows exclude the first one and whose columns exclude the second one, one can check that this minor is zero as shown in bold below:
\[  A = \left[ \begin{array}{ccc}
				1	&       	&	\\
				\bf 0	& 1		 &	\begin{array}{ccc} \bf 0 & \ldots & \bf 0 \end{array}\\
\begin{array}{c} \bf 0 \\ \vdots \\ \bf 0 \end{array}  &     	&	 \left [\begin{array}{ccc}
				        					       \bf 0 &\ldots &\bf 0 \\
									        \vdots &  &\vdots \\
									       \bf  0&\ldots   &\bf 0
				        						\end{array} \right]

			\end{array} \right].
		\]
Therefore $A$ has at least three 1's. In fact three 1's are sufficient, because the matrix with three 1's on the diagonal and the rest entries being zero has no $(n-1) \times (n-1)$ minor that is a zero matrix. The reason is that every $(n-1) \times (n-1)$ minor in an $n \times n$ matrix has at least $(n-2)$ entries on the diagonal, so it must contain at least one of the three 1's on the diagonal. Let's summarize the result in the following lemma. Moreover, we will generalize it to $n/2 < k \leq n$.
 
 \begin{lem}
 Suppose $A$ is an $n \times n$ matrix whose entries are either 0 or 1. Let an integer $1 \leq k \leq n$ be given. If any $k \times k$ minor of $A$ contains at least one entry of 1, then the number of 1's in $A$ is at least $2(n-k)+1$ for $ n/2 < k \leq n$. Moreover, if it is exactly equal to $2(n-k)+1$, $A$ is uniquely determined up to permutations of rows and columns.
 \end{lem}

\begin{proof}
The proof follows similarly as what we've shown for the case when $k=n-1$. 
First let's assume that $n$ is an even integer. Suppose for the sake of contradiction that the number of 1's in $A$ is less than or equal to $2(n-k)$. Since $n/2 +1 \leq k \leq n$, it implies that $0 \leq 2(n-k) \leq n-2$. If $2(n-k) = 0$, we return to the case when $k = n$ and the minimum number of 1's is proved to be 1. Let's assume that $n-k \geq 1$. Since $2(n-k) \leq n-2$, these 1's can be inductively fit into the first $2(n-k)$ rows and $2(n-k)$ columns of $A$ after appropriate permutations. Moreover, among these $2(n-k)$ rows and $2(n-k)$ columns, we may select $n-k$ rows and $n-k$ columns such that the $(n-k) \times (n-k)$ minor determined by them contains no 1 at all. This is due to a theorem which will be proved shortly (see Theorem~\ref{thm:n/2}). Adding the remaining $2k-n$ rows and $2k-n$ columns of zeros to these $n-k$ rows and $n-k$ columns, we find a trivial $k \times k$ minor, which leads to a contradiction. Therefore the number of 1's in $A$ is at least $2(n-k)+1$. 

Next consider the diagonal matrix $M$ whose first $2(n-k)+1$ diagonal entries are 1 and else are 0:
\[  M = \left[ \begin{array}{cc}
			\left[	\begin{array}{c}     \\
							     	I_{2(n-k)+1} \\
							     			\\
				\end{array} \right]
												&       		\\
				    							  	&    	0		
		\end{array}
		\right].\]			
Given a $k\times k$ minor, it must contain at least $2k-n$ diagonal entries, where $2k >n$ by hypothesis. Since $2k-n + 2(n-k)+1 = n+1$, the $k \times k$ minor has at least one 1. Thus $2(n-k)+1$ is the minimum number.

Finally we are ready to show the uniqueness. Suppose $A$ satisfies the hypothesis in the lemma and the number of 1's in $A$ is precisely $2(n-k)+1$. If these 1's are observed in different columns and different rows, they could be rearranged by permuting rows and columns if necessary to look exactly the same as the diagonal matrix $M$. Now suppose the 1's in $A$ are not in different rows or columns, we hope to reach a contradiction. 

{\it Case 1}: If they are not in different rows and also not in different columns, a $2(n-k) \times 2(n-k)$ block matrix could contain all of them. Applying Theorem~\ref{thm:n/2} again and the previous argument yields a contradiction. More precisely, if $n - k \geq 1$, $3(n-k)+1 > 2(n-k)+1$, it implies that there exists a $(n-k) \times (n-k)$ minor of zeros in the $2(n-k) \times 2(n-k)$ block matrix, thus $A$ has a $k \times k$ zero minor. 

{\it Case 2}: If they are in different rows but not in different columns (or in different columns but not in different rows), say they are inside a $(2(n-k)+1)\times p$ block matrix, where $1 \leq p \leq 2(n-k)$. When $p \leq n-k$, at least $k$ zero columns are present in $A$, which is impossible. Assume $(n-k)+1 \leq p \leq 2(n-k)$, let's denote the number of columns with only one 1 as $x$. Since each of the rest $p-x$ columns has more than one 1, the total number of 1's is at least $x + 2(p-x)$, that implies $ 2(n-k)+1 \geq x + 2(p-x) \Rightarrow x \geq 2p - 2(n-k)-1$,
which is greater than or equal to 1 (since $p-(n-k) \geq 1$ by our hypothesis). Therefore we may assume that the 1's belonging to these $x$ columns are on the diagonal, because they are also in different rows by our assumption. Furthermore, for the remaining 1's, since none of them are in the same row, we can group them in a nice way, namely into vertical column vectors of the form $(1, 1,\ldots, 1)^t$. As a result, the matrix $A$ can be put in the following form:
\[A=\left[\begin{array}{cccc|c}
		\left[	\begin{array}{cccc}
			&	&	\\
			& I_{x}	&\\
			&	&	
			\end{array}	
							\right] &					&		&	 	&					\\
								&\left[ \begin{array}{c}
								1 \\ \vdots\\ 1
								\end{array} \right]		&		&		&	0				\\
								&					& \ddots	&		&					\\
								&					&		& \left[ \begin{array}{c}
																1 \\ \vdots\\ 1
																\end{array} \right]  	&		\\ \hline
								&	0				&		&		&	0			
			\end{array} \right].
\]

What remains to show in order to obtain a contradiction is to select $k$ rows and $k$ columns so that the $k \times k$ minor determined by them is a zero matrix. For the rows, let's pick the last $k$ rows in the above form of $A$, in which the bottom $2k-n-1$ of them are rows of zeros, and the other $n+1 - k$ rows above them may contain the vertical column vectors of 1's, and some of the 1 in $I_x$. In fact, since the number of rows containing the vertical column vectors of 1's is $2(n-k)+1 - x$, we may have picked at most
\begin{equation*}
n+1 - k - [2(n-k)+1 - x] = x - (n-k)
\end{equation*}
of the bottom rows of $I_x$ if this integer is positive. 
Next for the columns, let's pick the first $p-(n-k)$ and the last $n-p$ of columns in the above form of $A$. Since the last $n-p$ columns are just zeros, there is no problem with them. Then for the first $p-(n-k)$ columns, since
\begin{equation}
x - (p-(n-k)) \geq p-((n-k)+1) \geq 0, \nonumber
\end{equation}
we need to make sure that they don't intersect with any of the previously selected rows, in particular the first $x-(n-k)$ of them. Since
\begin{equation*}
x-(n-k) + p-(n-k) = x + p-2(n-k) \leq x,
\end{equation*}
we won't introduce any 1 that is on the diagonal of $I_x$. Therefore we find a $k \times k$ minor that has only 0 entries. This is a contradiction. So we conclude that if the number of 1's in $A$ is precisely $2(n-k)+1$ and $A$ has no $k \times k$ minor of zeros, then A must be in the same form as $M$ up to permutations of rows and columns.  

For the case when $n$ is an odd integer, the proof is exactly the same, thus completing the proof of the lemma. 
\end{proof}

Henceforth we hope that the pattern continues to hold for $k \geq n/2 $. More precisely, let $\alpha(k,n)$ denote
the least number of 1's in an $n \times n$ matrix, in which no $k \times k$ minor is trivial. For $n/2 < k \leq n$, as $k$ decreases by 1, we've shown that $\alpha(k,n)$ increases by 2. If this continued to be true for $1 \leq k \leq n/2$, then for $k=1$ the least number would be 
\begin{equation*}
2(n-1)+1 = 2n-1,
\end{equation*}
which is strictly less than $n^2$, under the reasonable assumption that $n \geq 2$, and thus a contradiction. We deduce that  the pattern has to start changing somewhere between 1 and $n/2$. In fact, as we will see in the following theorem, the change occurs right at the ``middle point'' $k = n/2$ if $n$ is even, and $k =(n-1)/2$ if $n$ is odd, where the number $2k$ stops being greater than $n$. The easier case is when $n$ is an even integer and $k = n/2$, as $k$ decreases from $n/2 + 1$ to $n/2$, $\alpha(k,2k)$ increases by 2 plus $n/2$ (see Theorem~\ref{thm:n/2}). This is somewhat suggestive, because suppose this new pattern persists all the way down to $k=1$, then $\alpha(k,n)$ would be equal to:
\begin{equation*}
\alpha(k,n) = 2(n-(\frac{n}{2}+1))+1 + (2+\frac{n}{2})\frac{n}{2} = \frac{n^2}{4} + 2n -1,
\end{equation*}
which is close to $n^2$. On the other hand, when $n$ is an odd integer and $k = (n-1)/2$, the formula for $\alpha(k,n)$ turns out to be much more complicated than we might expect. So we will reserve the discussion for this case later.

\subsection{The case when $n$ is even and $ k = n/2$}

\begin{thm} \label{thm:n/2}
	Let $k$ be any positive integer $\geq 2$ and $n = 2k$. Given an $n \times n$ matrix $A$ with entries 0 or 1, suppose any $k \times k$ minor of $A$ contains at least one entry of 1, then the number of 1's in $A$ is at least $3k+1$. 
	Moreover, in the case that the number of 1's in $A$ is exactly $3k+1$, $A$ is unique up to permutations of rows and columns.
\end{thm}

\begin{proof}
	We can start with a few examples, because there is a pattern in them which helps to explain the proof of this theorem. 
	
	The first example is when $k = 2$, and a matrix with seven 1's satisfying the hypothesis looks like the following: 
\[	\left( \begin{array}{cccc}
			1 & 0 & 0 & 0 \\
			0 & 1 & 1 & 0 \\
			0 & 1 & 0 & 1 \\
			0 & 0 & 1 & 1
		\end{array}	\right).    \]
It's easy to check that any $2 \times 2$ minor is not a zero matrix. In fact, it suffices to check that for any $2 \times 4$ minor, there are at least three nonzero columns. 

	The next example is when $k = 3$, and a matrix with ten 1's satisfying the given condition is shown as follows:
\[	\left( \begin{array}{cccccc}
			1 & 0 & 0 & 0 & 0 & 0 \\
			0 & 1 & 0 & 0 & 0 & 0 \\
			0 & 0 & 1 & 1 & 0 & 0 \\
			0 & 0 & 1 & 0 & 1 & 0 \\
			0 & 0 & 0 & 1 & 0 & 1 \\
			0 & 0 & 0 & 0 & 1 & 1 
		\end{array}	\right).    \]
To verify that any $ 3 \times 3$ minor is not zero, we can simply look at all $3 \times 6$ minors and make sure that each has at least 4 nonzero columns. However, this method requires us to check 20 $3 \times 6$ minors. There is a better way than this. Let's first observe that if the first two rows are included in any $ 3 \times 6$ minor, since each of the rest rows has two 1's, the minor must have 4 nonzero columns. Therefore, what remains to check is the following two cases: 

{\it Case 1}:  If we pick one row from the first two rows and two rows from the rest four rows,  then it suffices to check that every $2 \times 4$ minor in the following matrix:
	\[	\left( \begin{array}{cccc}
			1 & 1 & 0 & 0 \\
			1 & 0 & 1 & 0 \\
			0 & 1 & 0 & 1 \\
			0 & 0 & 1 & 1
		\end{array}	\right)   \]
has at least three nonzero columns.	Here there are only six $2 \times 4$ minors, and it is easy to see that all six minors have at least three nonzero columns. 

{\it Case 2}: If we don't choose any of the first two rows, then it suffices to check that every $3 \times 4$ minor of the above matrix has at least four columns being nontrivial. After checking four $3 \times 4$ minors, this is also true. 

As you might guess from the above two examples, in the general case, a matrix with $3k+1$ number of 1's has the following pattern: 
\[	\left[ \begin{array}{ccc|ccccc}
			1 & 		 & 		   & 	&  		&  		&  		&\\
			   & 	 	\ddots & 	   &  	& 		&  		&  		& \\
			& 		 & 		1 &  	& 		&  		&  		& \\ \hline
			 &  		& 			   & 1 & 		1 & 		 &  	  & \\
			 &  		& 			   & 1 & 		0 & 		1 &  		&\\
			 & 		& 		           & 	& 		 & 		 \ddots  &   & \\
			 &  		& 			  & 	& 		&		1 & 0 & 1 \\
			 &  		& 			  & 	& 		&		0 & 1 & 1 
															\end{array}	\right].    \]
															
Let's describe the above matrix more precisely. Denote the matrix by $M$. For the first $k-1$ diagonal entries of $M$, we put 1's in them. What remains in $M$ is a $(k+1) \times (k+1)$ block matrix, call it $B$. The first row of $B$ has two 1's at the very beginning, and the last row has two 1's in the end. For each row in the middle, the two 1's are located next to the diagonal entry, one to the right and the other to the left. As a summary, $B$ is characterized by the following properties:
\[ \begin{array}{l}
b_{11} = b_{12} = 1 \nonumber\\
b_{k+1,k}= b_{k+1, k+1}=1 \nonumber\\
b_{j, j-1}=b_{j, j+1}=1, \text{ for } 2 \leq j \leq k, \nonumber
\end{array}  \]
and every other entry in $B$ is zero. Moreover, 
\[  M = \left[ \begin{array}{cccc}
				1	&       &	 		&\\
				        &     \ddots&      		&\\
				        &       &      	1	&\\
				        &	&			& \left [\begin{array}{ccc}
				        					         & & \\
									         & B & \\
									         &   &
				        						\end{array} \right]
			\end{array}
		\right]. \]															
															
Next we want to show that $M$ satisfies the hypothesis as given in the theorem. In analogy with the examples, it suffices to verify that there exist $m+1$ non-null columns in every $m \times (k+1)$ minor in the matrix $B$, where $1 \leq m \leq k$. Indeed, suppose we have selected $k-m$ of the first $k-1$ rows in $M$, which provides us with $(k-m)$ non-null columns, then the rest $m$ rows coming from $B$ must give at least $m+1$ columns that are not equal to zero. There are three cases to consider. 

{\it Case 1}: If a $m \times (k+1)$ minor $C$ does not pick up the first and last rows of $B$, then the positions of the 1's in $C$ is shown as below: 
\[
\left[ \begin{array}{ccccccc}
	\underline{i_1, i_1-1} &  		&  i_1, i_1+1  &			& & &\\
				      &\underline{i_2, i_2-1} & 		     &i_2, i_2+1  & & &\\
				      &			&		     & 			& \ddots& &\\
				      &			&		     &			& \underline{i_m,i_m-1}  & &\underline{i_m, i_m+1}
		\end{array}
					\right],
\]
where $i_j, i_j \pm 1$ means that the $(i_j, i_j \pm 1)$-entry of $B$ has a 1 for $1 \leq j \leq m$. Since $i_1 -1 < i_2 -1 < \cdots < i_m -1 < i_m+1$, $C$ has at least $m+1$ nonzero columns as underlined in the above matrix.  

{\it Case 2}: Suppose $C$ includes exactly one of the first and last rows of $B$, without loss of generality, we may assume that the first row of $B$ is in $C$. Then the 1's in $C$ are positioned in the following entries: 
\[ \left[ \begin{array}{ccccccc}
		\underline{1, 1} &  	\underline{1,2}				&  & & &\\
		i_2, i_2-1 	      &			& 	\underline{i_2, i_2+1}  & & &\\
				      &			&	 			& \ddots& &\\
				      &			&				& i_m,i_m-1  & &\underline{i_m, i_m+1}
		\end{array}
					\right].
\] Here we can again find at least $m+1$ distinct non-null columns in $C$, which belong to the 1st, 2nd, ($i_2 + 1$)th, \ldots, ($i_m + 1$)th columns of $B$ as underlined above.

{\it Case 3}: Lastly, when $C$ includes both the top and bottom rows of $B$, the pattern for the positions of 1's in $C$ looks like: 
\[ \left[ \begin{array}{ccccccc}
				1, 1&  	1,2	&  		& & &\\
			i_2, i_2-1&                & i_2, i_2+1  & & &\\
				      &			&      		& \ddots& &\\
				      &			&		& i_{m-1},i_{m-1}-1  & &i_{m-1}, i_{m-1}+1 \\
				      &			&		&  &k+1, k&k+1, k+1
		\end{array}  \right].\]
The previous observations don't work here, because at the first glance there seems to be `only' $m$ columns which are not equal to zero. However, this is just an illusion, because in the above matrix, by the way we present it, $i_2$ must be 2 and $i_{m-1}$ must be $k$. Thus we need to choose our presentation more carefully. 

Suppose $j$ is the first integer starting from 2 such that $i_j \neq j$, then the previous index matrix should look like: 
\[ \left[ \begin{array}{ccccccccccc}
	\underline{1, 1}&  \underline{1,2}&  		& & &&&&&\\
			2, 1&               &\underline{2, 3}  & & &&&&&\\
			&			&      		& \ddots& & &&&&\\
			 &			&		&j-1, j-2 & &j-1, j& &&&\\
			&			&		&     &  &\underline{i_j, i_j-1}& &i_j, i_j+1 &&\\
		&			&		& &&&&\ddots && \\				&			&		&&&&&\underline{k+1, k}&\underline{k+1,k+1}&
		\end{array}
		\right].
\]
Since there are 1's locating at the entries:$(1,1), (1,2), (2,3), \ldots, (j-2, j-1)$; $(i_j, i_j-1), \ldots, (i_{m-1}, i_{m-1}-1)$; and $(k+1, k), (k+1, k+1)$ (as underlined in the above matrix). In total, the number of distinct columns with at least one nonzero entry is at least: 
\begin{equation*}
(j-1) + (m-j) + 2 = m+1,
\end{equation*}
as desired. 

Then we want to show $3k+1$ is the minimum number, that is, $\alpha(k, 2k)=3k+1$. Given any $n \times n$ matrix $A$ with 0's or 1's satisfying the hypothesis, we can rearrange the rows of $A$ so that the total sum of 1's in each row is in an increasing order. Suppose one of the first $k-1$ rows has a sum at least 3, then the total number of 1's in $A$ is at least
\begin{equation*}
3 + 3(k+1) = 3k+6,
\end{equation*}
which is greater than $3k+1$. 

On the other hand, if each of the first $k-1$ rows has no more than three 1's, let's look at the total number of 1's in the first $k-1$ rows, and denote it by $S$. When $S < k-1$, each of the rest rows must have at least 3 1's, because recall that every $k \times (2k+1)$ minor in $A$ should contain at least $k+1$ nonzero columns. It follows that the number of 1's in $A$ is at least
\begin{equation*}
3(k+1) = 3k+3,
\end{equation*}
which is also greater than $3k+1$.
When $S = k-1$, each of the rest rows must have a minimum of two 1's; then the number of 1's is at least 
\begin{equation*}
(k-1)+2(k+1) = 3k+1.
\end{equation*}
This is the same as the number of 1's in $M$. Lastly, when $S > k-1$, it requires one of the first $k-1$ rows to have a sum of 1's at least 2; therefore, the number of 1's is at least
\begin{equation}
k+2(k+1)=3k+2,
\end{equation}
which is greater than $3k+1$. Thus $3k+1$ is the minimum number of 1's one need to put in $A$ in order to satisfy the hypothesis.

For the uniqueness part of the theorem, we observe that if the number of 1's in $A$ happens to be exactly $3k+1$, then $A$ must have $k-1$ rows with only one 1, and the rest $k+1$ rows with just two 1's. Through permutation of rows and columns if necessary, we may assume that the rows with only one 1 are the first $k-1$ rows in $A$ and the 1's are on the diagonal. Moreover, the rest of 1's must lie in the lower right corner of $A$, otherwise if one of them is below a 1 in the first $k-1$ rows, then we would obtain a $k \times 2k$ minor with not enough non-null columns, whence a $k \times k$ minor with all entries being zero exists in $A$. This is impossible. Therefore, the nonzero entries of the remaining $k+1$ rows of $A$ are forced to stay inside the $(k+1)\times (k+1)$ block matrix $C$ as indicated below:
\[  A = \left[ \begin{array}{cccc}
				1	&       &	 		&\\
				        &     \ddots&      		&\\
				        &       &      	1	&\\
				        &	&			& \left [\begin{array}{ccc}
				        					         & & \\
									         & C & \\
									         &   &
				        						\end{array} \right]
			\end{array}
		\right]. \]				
This is very much like the matrix $M$ that we've seen before, except that we are not sure yet whether $C$ is the same as $B$.

We claim that $C$ is equal to $B$, that is, the rows with two 1's in $A$ can be arranged in the same way as in $B$ up to permutations of rows and columns. First, let's verify that every column of $C$ has at least one 1's. Suppose one column is null, then pick any $k$ rows of $C$, their 1's are confined in $k$ columns, hence we find more than one $k \times 2k$ minors with less than $k+1$ nonzero columns. This is a contradiction to our hypothesis for $A$. Second, if some column has only one 1, look at the $k$ rows of $C$ which do not include this 1, it follows that they have no more than $k$ nonzero columns, so we conclude that every column of $C$ needs at least two 1's. Since there are totally $2(k+1)$ 1's in $C$ which has $(k+1)$ columns, we deduce that each column of $C$ possesses precisely two 1's. Finally, we are ready to arrange the 1's of $C$ in such a way that they will look exactly the same as in $B$. 

Through appropriate permutations, we may put two 1's in the first two entries of the first row of $C$ (see the matrix below). Next, we work on the second row of $C$. Since the first column of $C$ has another 1, we can put the first 1 in the first entry of the second row. Then, the other 1 in the second row cannot be right below the second 1 in the first row, otherwise we get two identical rows which is a contradiction to the fact that every two rows of $C$ need to have at least three nonzero columns. Therefore, we can put the other 1 in the third entry of the second row. We repeat the argument for the third row so the first 1 in the third row is put in the second column and the second 1 is put in the fourth column. We continue this process until the second to the last row, using the fact that the 1's in every $i$ rows of $C$ need to span $i+1$ columns for each $i$ from 2 to $k$ in order for $A$ to satisfy the hypothesis in the theorem. Finally for the last row of $C$, the two 1's must be in the lower right corner of the matrix $C$. As a result, $C$ is unique and is equal to $B$ up to permutations as follows:
\[  C = \left[ \begin{array}{ccccc}
				1	&      1 &	 0 & \ldots		&0\\
				 1       &      0&       	1&\ldots	&0\\
				  0      &       1&      	0&\dots	&0\\
				  \vdots &\vdots	  &	\vdots	&\ddots	&\vdots\\
				   0    &	0&	0& 	1&	1
		\end{array}
		\right]. \]	
This completes the proof.			

\end{proof}

\subsection{The case when $n$ is odd and $ k = (n-1)/2$}
Our next task is to find out what happens when $k=(n-1)/2$. As we've remarked earlier, this case is not as easy as the previous case when $n$ is even and $k = n/2$. 

Before proceeding, let's look at the following table: 
\[ \begin{array}{|c|c|c|c|c|c|c|c|c|c|c|c|}
\hline
 k \times k = &1 \times 1 &2 \times 2 & 3 \times 3 & 4 \times 4 & 5 \times 5 & 6 \times 6 & 7 \times 7&8\times8&9\times9	&10\times10 &11 \times 11 \\ \hline
 n\times n=&		&		&			&		&		&			&		&		&	&	&\\
 1 \times 1&	1	&		&			&		&		&			&		&		&	&	&\\
 2 \times 2 &\bf	4	& 1		  & 			&		&		&			&		&		&	& 	&\\
3 \times 3 & 	9	&3		   & 1 		& 		& 		& 			& 		&		&	&	& \\
4 \times 4 &	16	&\bf7		 & 3			&1		&		&			&		&		& 	&	&\\
5 \times 5 & 	25	&13		& 5			& 3		& 1		&			&		&		& 	&	&\\
6 \times 6 & 	36	&		&\bf10		&5		& 3		&1			&		&		& 	&	&\\
7 \times 7 & 	49	&		& 16			&7		&5		&3			&1		&		& 	&	&\\ 	
8 \times 8 & 	64	&		&			&\bf13	&7		&5			&3		&1	 	& 	&	&\\
9 \times 9	&	81	&		&			& 20		&9		&7			&5		&3		&1	&	&\\
10\times10&	100	&		&			&		&\bf16	&9			&7		&5		&3	& 1 	&\\ 
11 \times 11&	121	&		&			&		&23		&11			&9		&7		&5	&3	&1 \\
\hline \end{array}					
\]

In this table, each entry in the first column is an arbitrary $n \times n$ matrix, and each entry in the first row is a $k \times k$ minor. Each number in the table indicates the least number of 1's that we need to put into an $n \times n$ matrix so that every $k\times k$ minor contains at least one entry of 1. For example, in the row that begins with $4 \times 4$, the first number 16 says that in order to eliminate any zero $1 \times 1$ minor in a $4 \times 4$ matrix, we need at minimum of sixteen 1's. Next, the second number 7 (in bold) tells us that we need at least seven 1's in a $4 \times 4$ matrix so that no $2\times2$ minor in it is zero. Then, the third number 3 tells us that at least three 1's are needed in order for a $4 \times 4$ matrix to have no trivial $3 \times 3$ minor. Finally, the last number in that row is 1, which says that only one 1 is sufficient to eliminate any zero $4 \times 4 $ minor in a $4 \times 4$ matrix. 

Previously, we used $\alpha(k,n)$ to denote the smallest number of 1's that an $n \times n$ matrix with entries 0 or 1 needs, in order to have no zero $k \times k$ minor. So when $n$ is even and $k = n/2$, we've proved that
\begin{equation*}
\alpha(k, 2k) = 3k+1. 
\end{equation*}
Observing the table, we see that these numbers: $\alpha(1,2) = 4$, $\alpha(2,4)=7$, $\alpha(3,6)=10$, $\alpha(4,8)=13$, $\alpha(5,10)=16$, \ldots (shown in bold in the table) are the very first numbers breaking the sequences of odd integers in each column. In some sense, they are the {\it first} numbers for which the formula $\alpha(k, n) = 2(n-k)+1$, where $n/2 < k \leq n$, stops to apply. Then, in the case that $n$ is odd and $k = (n-1)/2$, the numbers $\alpha(k, 2k+1)$, for example $\alpha(1,3)=9$, $\alpha(2,5)=13$, $\alpha(3,7)=16$, $\alpha(4,9)=20$, etc.\ lie directly below the numbers 4, 7, 10, 13, etc.\ , so they can be viewed as the {\it second} numbers breaking the pattern. The formula for the {\it second} numbers is harder to discover than for the {\it first} numbers. In fact, in the following, step-by-step trials are described during the search for a precise formula for $\alpha(k, 2k+1)$. Eventually, we find that it is `impossible' to write down an explicit one, at least using finitely many linear equations, because the candidate formulae converge to an asymptotic formula for $\alpha(k, 2k+1)$. 

{\bf Trial 1}: Following the nice pattern shown in Theorem~\ref{thm:n/2}, we naturally ask whether it can be generalized or not. First, we make an observation for the case when $k=2$. The minimum number $\alpha(2,5)$ can be easily calculated to be 13, and two examples are shown as follows:
\[ \begin{array}{cc}

 \left[ \begin{array}{ccccc}
1&0&0&0&0\\
0&1&1&1&0\\
0&1&1&0&1\\
0&1&0&1&1\\
0&0&1&1&1
\end{array} 
\right] &      \left[ \begin{array}{ccccc}
1&1&0&0&0\\
0&0&1&1&0\\
1&0&0&1&1\\
0&1&1&0&1\\
1&0&1&0&1
\end{array} 
\right].          \\
\end{array}
\]
Immediately it implies that uniqueness no longer holds. However, the matrix on the left looks very similar to the matrix $M$ in Theorem~\ref{thm:n/2}, because it has 1's on the diagonal and 3's in the remaining rows. Therefore, the following pattern seems to work:
\[ \left[ \begin{array}{ccc|ccccccc}
1&&&&&&&&&\\
&\ddots&&&&&&&&\\
&&1&&&&&&&\\ \hline
&&&1&1&1&&&&\\
&&&1&1&&1&&&\\
&&&1&&1&&1&&\\
&&&&&&&\ddots&&\\
&&&&&1&&1&&1\\
&&&&&&1&&1&1\\
&&&&&&&1&1&1
\end{array}
\right].
\]

We can show that the pattern does produce a matrix whose $k\times k$ minors are never zero, however it does not give the minimum number.

\begin{lem} \label{lem:4k}
Let $k$ be any positive integer and $n = 2k+1$. Given an $n \times n$ matrix $A$ with entries 0 or 1, suppose any $k \times k$ minor of $A$ contains at least one entry of 1, then the minimum number $\alpha(k, 2k+1)$ of 1's in $A$ is at most $4k+5$. 
\end{lem}

\begin{proof}
Let's give a more precise definition of the above matrix. Denote the matrix by $N$ and assume that $k\geq 2$ (the case when $k=1$ is obvious), then the first $k-1$ diagonal entries of $N$ are 1. What remains in $N$ is a $(k+2) \times (k+2)$ block matrix, let's call it $C$. The first and second rows of $C$ are shown as above, as well as the last and the second to last rows. For each row in between, there are three 1's with one of them on the diagonal and the other two staying one entry away from the diagonal. That is to say, 
\[ \begin{array}{c}
c_{11} = c_{12} = c_{13}=1 \nonumber\\
c_{21} = c_{22} = c_{24}=1 \nonumber\\
c_{k+1, k-1} = c_{k+1, k+1} = c_{k+1, k+2} = 1 \nonumber\\
c_{k+2, k} = c_{k+2,k+1}= c_{k+2, k+2}=1 \nonumber\\
c_{j, j-2}=c_{j,j} = c_{j, j+2}=1, \text{ for } 3 \leq j \leq k, \nonumber
\end{array}  \]
and  the other entries of $N$ are zero. Moreover, 
\[ N = \left[ \begin{array}{cccc}
				1	&       &	 		&\\
				        &     \ddots&      		&\\
				        &       &      	1	&\\
				        &	&			& \left [\begin{array}{ccc}
				        					         & & \\
									         & C & \\
									         &   &
				        						\end{array} \right]
			\end{array}
		\right], \]	
which has totally $4k+5$ 1's. 

In analogy with the proof for Theorem~\ref{thm:n/2}, it suffices to show that every $m$ rows of $C$ have at least $m+2$ nonzero columns. There are six cases to consider here.

{\it Case 1}: If a $m \times (k+2)$ minor $D$ contains only rows in the middle, that is, row 3, \ldots, row $k$, then the positions of the 1's in $D$ are illustrated as follows:
\[
\left[ \begin{array}{ccccccccc}
	\underline{i_1, i_1-2} &    		&\underline{i_1, i_1}	&  		     &	i_1, i_1+2  &			& & &\\
				      &i_2, i_2-2 &                  &\underline{i_2, i_2}	&			     &i_2, i_2+2  & & &\\
				      &			&		     & 			& \ddots& 			&		&&\\
				      &			&		     &			& i_m,i_m-2  &		&\underline{i_m, i_m} &	&\underline{i_m, i_m+2}
		\end{array}
					\right],
\]
where the underlined entries give $m+2$ nonzero columns in $D$. 

{\it Case 2}: $D$ has exactly one row from either the top or the bottom two rows of $C$. Without loss of generality, we may assume that it is from the top two rows. It follows that the 1's in $D$ are also located in at least $m+2$ different columns as indicated below:
\[
\left[ \begin{array}{cccccccc}
				\underline{1,1}	&\underline{1,2}	&		1,3&	&		&&&\\		
				 (   \underline{2,1}  &\underline{2,2}	&  		     &	2,4) &			& & & \\
				      i_2, i_2-2 &                  &\underline{i_2, i_2}	&			     &i_2, i_2+2  & & &\\
				      			&		     & 			& \ddots& 			&		&&\\
				      			&		     &			& i_m,i_m-2  &		&\underline{i_m, i_m} &	&\underline{i_m, i_m+2}
		\end{array}
					\right].
\]
Here are two separate subcases, and we've included the second row in parentheses. Each subcase gives at least $m+2$ nonzero columns.

{\it Case 3}: $D$ contains the first (or the bottom) two rows. Then the positions of 1's in $D$ are shown as follows:
\[
\left[ \begin{array}{cccccccc}
				\underline{1,1}	&\underline{1,2}	&\underline{1,3}&	&		&&&\\		
				  2,1  &2,2	&  		     &\underline{	2,4 }&			& & & \\
				      i_3, i_3-2 &                  &i_3, i_3	&			     &\underline{i_3, i_3+2}  & & &\\
				      			&		     & 			& \ddots& 			&		&&\\
				      			&		     &			& i_m,i_m-2  &		&i_m, i_m &	&\underline{i_m, i_m+2}
		\end{array}
					\right],
\]
which gives at least $m+2$ nonzero columns in $D$. 	
		
{\it Case 4}: $D$ contains exactly one row from each of the first and last two rows of $C$. This is similar to case 2, which is shown below.
\[
\left[ \begin{array}{cccccccc}
\underline{1,1}	&\underline{1,2}	&		1,3&	&		&&&\\		
( \underline{2,1}  &\underline{2,2}	&  		     &	2,4) &			& & & \\
 &                  &\underline{i_2, i_2}	&			     & & & &\\
&		     & 			& \ddots& 			&		&&\\
&		     &			&   &		&\underline{i_{m-1}, i_{m-1} }&	&\\
&			&		&				&(k+1,k-1	&					&\underline{k+1, k+1}	&k+1, k+2)\\
&			&		&				&		&		k+2, k		&		k+2, k+1		&\underline{k+2, k+2	}
\end{array} \right].
\]
Here are four separate subcases, but each gives at least $m+2$ non-null columns.

{\it Case 5}: $D$ contains the first two rows and exactly one row from the bottom two rows (or vice versa). Suppose $j$ is the first integer starting from 3 such that $i_j \neq j$, then the index matrix becomes: 
\[ \left[ \begin{array}{cccccccccc}
\underline{1,1} & \underline{1,2}& \underline{1,3}&&&&&&& \\
&&&\ddots&&&&&&\\
&&&&\underline{j-1,j+1}&&&&& \\
&&&&i_j, i_j			&&&&&\\
&&&&&\underline{i_{j+1}, i_{j+1}}&&&&\\
&&&&&&\ddots&&&\\
&&&&&&&\underline{i_{m-1}, i_{m-1}}&&\\
&&&&&&&&(\underline{k+1,k+1}&\underline{k+1, k+2})\\\
&&&&&&&&\underline{k+2, k+1}&\underline{k+2,k+2}
\end{array}
\right], 
\]
whose underlined indices provide $m+2$ nonzero columns.

{\it Case 6}: $D$ contains both the top and the bottom two rows. Suppose $j$ is the first integer such that $i_j \neq j$. On the one hand, if $i_j \geq j+2$, then the positions of 1's in $D$ are in at least $m+2$ different columns shown as follows:
\[ \left[ \begin{array}{cccccccccc}
\underline{1,1} & \underline{1,2}& \underline{1,3}&&&&&&& \\
&&&\ddots&&&&&&\\
&&&&\underline{j-1,j+1}&&&&& \\
&&&&&\underline{i_j, i_j}&&&&\\
&&&&&&\ddots&&&\\
&&&&&&&\underline{i_{m-1}, i_{m-1}}&&\\
&&&&&&&&\underline{k+1,k+1}&\\\
&&&&&&&&k+2, k+1&\underline{k+2,k+2}
\end{array}
\right].
\]
On the other hand, suppose $i_j=j+1$, then there must be another gap between $i_{j+1}$ and $i_{m-1}$, since we select $m \leq k$ rows out of $k+2$ rows. Suppose $p$ is the next integer starting from $j$ such that $i_p \neq p+1$, then we obtain the following matrix:
\[ \left[ \begin{array}{ccccccccc}
\ddots&&&&&&&&\\
&\underline{j-1,j+1}&&&&&&& \\
&&&\underline{j+1,j+3}&&&&&\\
&&\underline{j+2, j+2}&&&&&& \\
&&&\ddots&&&&&\\
&&&&\underline{p, p+2}&&&&\\
&&&&i_p, i_p&&&&\\
&&&&&\underline{i_{p+1}, i_{p+1}}&&&\\
&&&&&&\ddots&&\\
&&&&&&&\underline{k+1,k+1}&\\
&&&&&&&&\underline{k+2,k+2}
\end{array}
\right],
\]
where we've omitted the first $j-2$ rows because they are the same as before. Again there are at least $m+2$ nonzero columns, thus completing the proof.  	
\end{proof}

The formula $4k+5$ starts failing at $k=3$ because $\alpha(3, 7) = 16 < 17$ based on the following example, which is in fact unique up to permutations: 
\[\left[
\begin{array}{c|cc|cc|cc}
1&&&&&&\\ \hline
&1&1&&&&\\
&&&1&1&&\\
&&&&&1&1\\ \hline
&1&&&1&&1\\
&&1&1&&&1\\
&&1&&1&1&
\end{array}
\right].
\]
Moreover, the formula $4k+5$ also fails at $k=4$ because $\alpha(4,9)=20 < 21$ based on the following two examples, which imply that uniqueness no longer holds:
\[\begin{array}{cc}
\left[ \begin{array}{cc|cc|cc|cc|c}
1&&&&&&&&\\
&1&&&&&&&\\ \hline
&&1&1&&&&&\\
&&&&1&1&&&\\
&&&&&&1&1&\\ \hline
&&1&&&1&&1&\\
&&&1&1&&{\bf1}&&\\
&&&1&&{\bf1}&&&1\\
&&&&{\bf1}&&&1&1
\end{array}
\right], & \left[\begin{array}{c|ccc|ccc|cc}
1&&&&&&&&\\ \hline
&1&1&&&&&&\\
&&1&1&&&&&\\
&&&&1&1&&&\\
&&&&&1&1&&\\
&&&&&&&1&1\\ \hline
&&1&&1&&&1&\\
&&&1&&&1&1&\\
&1&&&&1&&&1
\end{array}\right].
\end{array}
\]

{\bf Trial 2}: Examining the first few examples of $\alpha(k, 2k+1)$ in the table, we observe that when $k$ is odd, the minimum numbers are 9, 16, 23, and when $k$ is even, the minimum numbers are 13, 20. We conjecture that the formula for $\alpha(k, 2k+1)$ is equal to:
\[ \alpha(k, 2k+1) = \left\{ \begin{array}{ll}
					\frac{7k+11}{2} & \mbox{if $k$ is odd}\\
					\frac{7k+12}{2} & \mbox{if $k$ is even}
					\end{array}
\right.
\]
However, like $4k+5$, this is proved to be only an upper bound for $\alpha(k, 2k+1)$ as shown in the following lemma.

\begin{lem} \label{lem:7k/2}
Let $k$ be any positive integer and $n = 2k+1$. Given an $n \times n$ matrix $A$ with entries 0 or 1, suppose any $k \times k$ minor of $A$ contains at least one entry of 1, then the minimum number $\alpha(k, 2k+1)$ of 1's in $A$ is at most
\[\left\{ \begin{array}{cc}
\frac{(7k+11)}{2} & \mbox{if $k$ is odd} \\
\frac{(7k+12)}{2} & \mbox{if $k$ is even.}
\end{array}
\right.
\]
\end{lem}

\begin{proof}
First, let's assume that $k$ is an odd integer. Based on the above example for $k=3$, we generalize it the following matrix $L$ of 0's and 1's:
\[ L = \left[ \begin{array}{ccc|cc|cc|cc|cc}
	1&&&&&&&&&&\\
	&\ddots&&&&&&&&&\\
	&&1&&&&&&&&\\ \hline
	&&&1&1&&&&&&\\
	&&&&&1&1&&&&\\
	&&&&&&&\ddots&&&\\
	&&&&&&&&&1&1\\ \hline
	&&&{\bf1}&&&&&&&1\\
	&&&&1&{\bf1}&&&&&1\\
	&&&&1&&1&\vdots&&&\\
	&&&&&&1&&&{\bf1}&
	\end{array}
\right].
\]

More precisely, the first $k-2$ rows of $L$ consist of diagonal entries of 1. For the next $(k+3)/2$ rows, each of them has two adjacent 1's that are to the right of the two 1's in the preceding row.  Then for the last $(k+3)/2$ rows, each has three 1's, and they are arranged as follows:
\begin{enumerate}
\item There are three 1's underneath every pair of adjacent 1's in the rows with two 1's. 
\item The three 1's are always in the pattern of $\left(\begin{array}{cc} 1 & \\ & 1\\ & 1 \end{array}\right)$, with one 1 on the left and two 1's on the right which are one entry behind the 1 on the left. If there is not enough room on the right, start over to the first row with three 1's (see the example (on the left) when $k=4$).
\item The 1's on the left side in each triple form a diagonal pattern (shown in bold in the matrix $L$). 
\end{enumerate}

Let's verify that every $k \times (2k+1)$ minor in $L$ has at least $k+2$ nonzero columns. First if we only select rows with one or two 1's, say $m$ from the first $k-2$ rows and $k-m$ from the next $(k+3)/2$ rows, where $ m \leq k-2$, then there are at least $m + 2(k-m) \geq 2k -(k-2) = k+2$ columns that are nontrivial.

Next if we select one row with three 1's, since $3 + (k-1) = k+2$, it suffices to check that the 1's in each row of the first $(k-2) + (k+3)/2$ rows provide at least one new column, besides the columns where the three 1's belong to. It is obviously true for the first $k-2$ rows. For each pair of 1's in the next $(k+3)/2$ rows, since the three 1's beneath them come from three different rows, this is also true. Then we proceed by induction on the number $l$ of rows with three 1's in a $k \times (2k+1)$ minor. Here we need to look at three cases.

{\it Case 1}: When $l = (k+3)/2$, we include all rows with three 1's that distribute over $k+3$ distinct columns, which is more than $k+2$. 

{\it Case 2}: When $l=(k+3)/2 -1$, we include all rows with three 1's except for one. According to case 1, it suffices to show that deleting one row only results in a loss of one column, that is, the 1's in the  rest $l$ rows are over $k+2$ distinct columns. This is true by our construction, because each row has only one 1 that is on the left-hand side while the other two have another 1 beneath them.

{\it Case 3}: When $l \leq (k+3)/2-2$, first we can prove by induction that the 1's in these $l$ rows distribute over at least $2l+1$ columns. Suppose not, at least $l$ 1's are above another 1, this is only possible for the first 1 on the right in the pattern of $\left(\begin{array}{cc} 1 & \\ & 1\\ & 1 \end{array}\right)$. Since each row has one and only one such 1, it follows that the $l$ rows must be consecutive. If we look at the last one of them, either it is not the bottom row which implies there is another 1 below it, or it is the bottom one which implies there is another 1 above it. In either case, an additional row is needed because we assume that $l \leq (k+3)/2-2$. 

Second, if the 1's in the $l$ rows span over $2l+i$ columns, for $1 \leq i \leq l$, we do the calculation: $2l+i +(k-l) = (k+2) +(l+i-2)$. It is sufficient to prove that every one of the first $(k-2)+(k+3)/2$ rows supplies at least one new column that is not one of the columns where the 1's in the $l$ rows belong to, with at most $l+i-2$ exceptions. Since each 1 in the first $k-2$ rows gives a new column, let us focus on the pairs of  adjacent 1's in the middle rows. Suppose for the sake of contradiction, there are $l+i-1$ pairs lying over the $l$ rows. 

When $i=1$, It follows that in the triple of $\left(\begin{array}{cc} 1 & \\ & 1\\ & 1 \end{array}\right)$ under each pair, one must include the 1 on the left-hand side and one of the two 1's on the right-hand side. Since each pair needs one 1 to the left and there are $l$ of them, the $l$ rows with three 1's are thus determined. Furthermore, we observe that in the row of each 1 that is on the left-hand side of $\left(\begin{array}{cc} 1 & \\ & 1\\ & 1 \end{array}\right)$, there are two additional 1's to its left except for the first and second rows in the last $(k+3)/2$ rows. Since the $l$ pairs of 1's already span over $2l$ columns and the $l$ rows distribute over only $2l+1$ columns, the first one of them must start with the first or second one of the last $(k+3)/2$ rows, and the rest of them need to be consecutive in order to avoid introducing any new column. However, the last one of them requires an additional row beneath it in order to fit in an extra 1 to the right, thus reaching a contradiction. When $i \geq 2$, the proof is analogous and we leave it as an exercise. Thus we complete the proof for the case when $k$ is an odd integer. 

Next, if $k$ is an even integer, we adjust the above matrix $L$ to a new matrix $L'$ in the following way. The first $k-2$ rows of $L'$ are the same as $L$ which consist of diagonal entries of 1. Likewise for the next $(k+2)/2$ rows of $L'$, each of them has two adjacent 1's that are to the right of the two 1's in the preceding row. However, unlike $L$ the last column of $L'$ is still an empty column so far. Then for the remaining $(k+2)/2+1$ rows, each has three 1's, and they are arranged in the same way as in $L$, except for the last three triples of $\left(\begin{array}{cc} 1 & \\ & 1\\ & 1 \end{array}\right)$. One example (on the left) is shown before the lemma for $k=4$. The three 1's shown in bold serve as the 1's to the left in the last three triples. They are intertwined together. Indeed, if we look from right to left, the first triple is made of the last column and the first bold 1 to the right. The second triple is made of the second to last column and the second bold 1 to the right. What's remaining is for the third triple. In general, the rows with three 1's are arranged as follows:
\[L'=\left[ \begin{array}{c|cc|cc|c|cc|cc|cc|c}
\ddots&&&&&&&&&&&&\\ \hline
&1&1&&&&&&&&&&\\
&&&1&1&&&&&&&&\\
&&&&&\ddots&&&&&&&\\
&&&&&&1&1&&&&&\\
&&&&&&&&1&1&&&\\
&&&&&&&&&&1&1&\\ \hline
&{\bf1}&&&&&&&&1&&1&\\
&&1&&&&&&1&&&&\\
&&1&{\bf1}&&&&&&&&&\\
&&&& 1&\ddots&&&\vdots&\vdots&&\vdots&\\
&&&& 1&&&&&&&&\\
&&&&   &&{\bf1}&&&&{\bf1}&&\\
&&&&   &&  &1&&{\bf1}&&&1\\
&&&&   &&  &1&{\bf1}&&&	1&1	
\end{array}
\right].
\]
We omit the proof, because it is similar as before. 
\end{proof}

The formula $(7k+11)/2$ or $(7k+12)/2$ still fails to formulate the minimum number $\alpha(k, 2k+1)$. The first counterexample is discovered for $k=19$ where we can find a matrix $A$ with $71$ 1's, which is less than 72 as predicted by the formula. The counterexample is a $39 \times 39$ matrix \[A=\left[\begin{array}{c|c}
			P_1 & P_2 \\ \hline
			0	& P_3
			\end{array}\right],\] whose the three block matrices $P_1$, $P_2$, and $P_3$ are shown as below:
\[\begin{array}{c}
P_1=\left[ \begin{array}{r|c}
\overbrace{\begin{array}{ccc} 1&&\\&\ddots&\\&&1 \end{array}}^{15} & \\ \hline
1&1\\
\end{array}
\right],
\end{array}
\]
\[\begin{array}{c}
P_2=\overbrace{\left[\begin{array}{ccc|ccc|ccc|cc}
1&1&&&&&&&&&\\
&1&1&&&&&&&\\
&&&1&1&&&&&&\\
&&&&1&1&&&&&\\
&&&&&\ddots&&&&&\\
&&&&&&1&1&&&\\
&&&&&&&1&1&&\\
&&&&&&&&&1&1\\
\end{array}\right]}^8, 
\end{array}
\]
\[\begin{array}{c}
 P_3 =\overbrace{ \left[ \begin{array}{c|c|c|c|c|c|c|c}
*&*&*&&&&&\\
*&*&&*&&&&\\
*&&*&&*&&&\\
&*&&*&&*&&\\
&&*&&*&&*&\\
&&&*&&*&&*\\
&&&&*&&*&*\\
&&&&&*&*&*
\end{array}
\right]}^8.
\end{array}
\]

In $P_1$ there is a $15 \times 15$ identity matrix followed by a row of two 1's in which the first 1 is right below the last 1 in the identity matrix. In $P_2$, we divide the columns into 8 groups with the first seven being $\left(\begin{array}{ccc}1 & 1& \\ & 1 & 1\end{array} \right)$ and the eighth being $(\begin{array}{cc} 1&1 \end{array})$. The groups don't stack over each other.  The matrix $P_3$ is only schematic. Each * represents a 1, so every row has three 1's. However, each column in $P_3$ actually corresponds to three (or two) columns of a group in $P_2$, so $P_3$ only tells you how to put three 1's below each group in $P_2$. In fact, for each of the first seven columns, the three * span over three distinct columns, and in the last column, they span over two columns. One can check that every $19 \times 19$ minor of $A$ has at least one entry of 1. 

Analogously, a second counterexample is discovered for $k=20$ where where we can find a $41 \times 41$ matrix with 75 1's, which is less than 76 as predicted by the formula. The example is similar to  the one for $k=19$, where $P_3$ is the same and $P_1$ and $P_2$ are illustrated as follows:
\[\begin{array}{cc}
P_1=\left[ \begin{array}{r|c}
\overbrace{\begin{array}{ccc} 1&&\\&\ddots&\\&&1 \end{array}}^{16} & \\ \hline
1&1\\
\end{array}
\right], &
P_2=\overbrace{\left[\begin{array}{ccc|ccc|ccc}
1&1&&&&&&&\\
&1&1&&&&&\\
&&&1&1&&&&\\
&&&&1&1&&&\\
&&&&&\ddots&&&\\
&&&&&&1&1&\\
&&&&&&&1&1\\
\end{array}\right]}^8. 
\end{array} 
\]

Moreover, a third counterexample is discovered for $k=21$ where we find a $43 \times 43$ matrix with 78 1's, which is less than 79 as predicted by the formula. The example is also similar as before, where $P_3$ is the same and $P_1$, $P_2$ are illustrated as follows:  
\[\begin{array}{cc}
P_1=\overbrace{\left[ \begin{array}{r}
\begin{array}{ccc} 1&&\\&\ddots&\\&&1 \end{array} 
\end{array}
\right]}^{17}, &
P_2=\overbrace{\left[\begin{array}{ccc|ccc|ccc|cc}
1&1&&&&&&&&&\\
&1&1&&&&&&&\\
&&&1&1&&&&&&\\
&&&&1&1&&&&&\\
&&&&&\ddots&&&&&\\
&&&&&&1&1&&&\\
&&&&&&&1&1&&\\
&&&&&&&&&1&1
\end{array}\right]}^9.
\end{array}
\]

More generally, we show that the following statement is true.
\begin{lem} \label{lem:10k/3}
Let $k$ be any positive integer and $n = 2k+1$. Given an $n \times n$ matrix $A$ with entries 0 or 1, suppose any $k \times k$ minor of $A$ contains at least one entry of 1, then the minimum number $\alpha(k, 2k+1)$ of 1's in $A$ is at most
\[\left\{ \begin{array}{cc}
		\frac{10k+24}{3} & \mbox{if $k \equiv 0$ (mod 3)}\\
		\frac{10k+23}{3} & \mbox{if $k \equiv 1$ (mod 3)}\\
		\frac{10k+25}{3} & \mbox{if $k \equiv 2$ (mod 3)}
	\end{array}
\right.
\]
\end{lem}

\begin{proof}
The previous three examples are the models for the general cases.

{\it Case 1}: Suppose $k \equiv 0$ (mod 3), then the example for $k = 21$ can be generalized as follows: 
\[\begin{array}{cc}
P_1=\overbrace{\left[ \begin{array}{r}
\begin{array}{ccc} 1&&\\&\ddots&\\&&1 \end{array} 
\end{array}
\right]}^{k-4}, &
P_2=\overbrace{\left[\begin{array}{ccc|ccc|ccc|cc}
1&1&&&&&&&&&\\
&1&1&&&&&&&\\
&&&1&1&&&&&&\\
&&&&1&1&&&&&\\
&&&&&\ddots&&&&&\\
&&&&&&1&1&&&\\
&&&&&&&1&1&&\\
&&&&&&&&&1&1
\end{array}\right]}^{(k+6)/3},
\end{array}
\]
and $P_3$ has $(k+6)/3$ columns of three * which follow similar patterns as before:
\[\begin{array}{c}
 P_3 =\overbrace{ \left[ \begin{array}{c|c|c|c|c|c|c}
*&*&*&&&&\\
*&*&&&&&\\
*&&*&&&&\\
&*&&\ddots&*&&\\
&&*&&&*&\\
&&&&*&&*\\
&&&&&*&*\\
&&&&*&*&*
\end{array}
\right]}^{(k+6)/3}.
\end{array}
\]

{\it Case 2}:  Suppose $k \equiv 1$ (mod 3), then the example for $k=19$ can be generalized in a similar way: 
\[\begin{array}{cc}
P_1=\left[ \begin{array}{r|c}
\overbrace{\begin{array}{ccc} 1&&\\&\ddots&\\&&1 \end{array}}^{k-4} &\\ \hline
1&1
\end{array}
\right], &
P_2=\overbrace{\left[\begin{array}{ccc|ccc|ccc|cc}
1&1&&&&&&&&&\\
&1&1&&&&&&&\\
&&&1&1&&&&&&\\
&&&&1&1&&&&&\\
&&&&&\ddots&&&&&\\
&&&&&&1&1&&&\\
&&&&&&&1&1&&\\
&&&&&&&&&1&1\\
\end{array}\right]}^{(k+5)/3}, 
\end{array}
\]
and $P_3$ has $(k+5)/3$ columns following the same pattern as above.

{\it Case 3}: Suppose $k \equiv 2$ (mod 3), then the example for $k=20$ can be generalized as follows:
\[\begin{array}{cc}
P_1=\left[ \begin{array}{r|c}
\overbrace{\begin{array}{ccc} 1&&\\&\ddots&\\&&1 \end{array}}^{k-4} & \\ \hline
1&1\\
\end{array}
\right], &
P_2=\overbrace{\left[\begin{array}{ccc|ccc|ccc}
1&1&&&&&&&\\
&1&1&&&&&\\
&&&1&1&&&&\\
&&&&1&1&&&\\
&&&&&\ddots&&&\\
&&&&&&1&1&\\
&&&&&&&1&1\\
\end{array}\right]}^{(k+4)/3}, 
\end{array} 
\]
and $P_3$ has $(k+4)/3$ columns following the same pattern as above.
\end{proof}

{\bf Trial 3}: We observe that the coefficients of $k$ in the above formulas are 4, $\frac{7}{2}$, $\frac{10}{3}$, which are rational numbers in the form of $(3a+1)/3a$, where $a$ is an integer. As $a$ increases, (3a+1)/3a gradually decreases and it eventually converges to 3. 

\begin{thm}
Let $k$ be any positive integer and $n = 2k+1$. Given an $n \times n$ matrix $A$ with entries 0 or 1, suppose any $k \times k$ minor of $A$ contains at least one entry of 1, then given a positive integer $a$, the minimum number $\alpha(k, 2k+1)$ of 1's in $A$ is at most:
\begin{equation}
\frac{(3a+1)k+C}{a}, \nonumber
\end{equation} 
where the constant $C$ depends on the different congruence class of $k$ modulo $a$. In particular, $\alpha(k, 2k+1)$ converges to $3k$ (plus a constant) as $k$ goes to infinity. 
\end{thm}

\begin{proof}
We observe that when $a = 2$, $(\begin{array}{cc} 1 & 1 \end{array})$ appears in groups, each of which lies over three 1's (see $L$ and $L'$ in Lemma~\ref{lem:7k/2}); when $a = 3$, $\left(\begin{array}{ccc}1 & 1& \\ & 1 & 1\end{array} \right)$ appears in groups each also lying over three 1's (see $P_2$ and $P_3$ in Lemma~\ref{lem:10k/3}); so for an arbitrary $a$, let's consider the following $(a-1) \times a$ matrix $Q_a$:
\[Q_a=\left[ \begin{array}{cccccc}
1&1&&&&\\
&1&1&&&\\
&&1&1&&\\
&&&\ddots&\\
&&&&1&1
\end{array}
\right].
\]   
Then the  block matrix $P_2$ in Lemma~\ref{lem:10k/3} becomes:
\[P_2=\left[\begin{array}{r|c} \overbrace{
\begin{array}{cccc}
Q_a&&&\\
&Q_a&&\\
&&\ddots&\\
&&&Q_a\\
\end{array}}^{m} & \\ \hline
& R 
\end{array}
\right]
\]
in which $m$ is an integer that is approximately $(k+2a-1)/a$, and $R$ is a remainder, depending on the different congruence classes of $k$ (mod $a$). Likewise, $P_1$ and $P_3$ become 
\[\begin{array}{cc}
P_1=\left[ \begin{array}{r|c}
\overbrace{\begin{array}{ccc} 1&&\\&\ddots&\\&&1 \end{array}}^{k-2a+2} & \\ \hline
&R\\
\end{array}
\right], &
 P_3 =\overbrace{ \left[ \begin{array}{c|c|c|c|c|c|c}
*&*&*&&&&\\
*&*&&&&&\\
*&&*&&&&\\
&*&&\ddots&*&&\\
&&*&&&*&\\
&&&&*&&*\\
&&&&&*&*\\
&&&&*&*&*
\end{array}
\right]}^{m}.
\end{array}
\]
The total number of 1's in $P_1$, $P_2$, and $P_3$ is approximately
\begin{equation}
k-2a+2 + 2(a-1)m+ 3m = \frac{(3a+1)k +C}{a}, \nonumber
\end{equation}
for some constant $C$. The rest of the proof is analogous as before. In particular, when as $k$ becomes larger and larger, $\alpha(k, 2k+1)$ converges to $3k$ plus a constant. 
\end{proof}

\section{Conclusion}
In this paper we study a problem of Erd\"{o}s concerning lattice cubes. We want to find the maximal number of vertices one can select from an $N \times N \times N$ lattice cube so that no eight corners of a rectangular box inside the grid are selected simultaneously. Efforts have been made by mathematicians such as Erd\"{o}s and Katz to estimate how big this set might be. A sharp upper bound of $N^{\frac{11}{4}}$ has been conjectures, but no example that large has been found so far to confirm this conjecture. Katz and etc.\ have found an example that is $O(N^{8/3})$.

This paper starts with investigating small examples, such as $N = 2, 3, 4$, for which the maximum numbers are 7, 22, and 47, respectively. The first two numbers are realized by unique configurations up to permutational, rotation, and reflectional symmetries. On contrast, the third number are obtained through two distinct configurations, which are not equivalent to each other. So a maximal set of vertices does not have to be unique. Furthermore, we find that the method of exhaustion quickly meets its limitation when $N=5$, so we look for another way to approach this question. One way is to study an equivalent two-dimensional problem in terms of matrices and hope that a pattern could be discovered and generalized to its three-dimensional equivalence. 

Since an $n \times n$ grid is also an $n \times n$ matrix, we can rephrase the question as: what is the minimum number of 1's one can put in an $n \times n$ matrix with entries 0 or 1, such that every $2 \times2$ minor contains at least one entry of 1? Moreover, we ask a more general question: given $1 \leq k \leq n$, what is the minimum number $\alpha(k,n)$ of 1's such that any $k \times k$ minor contains at least one entry of 1? First, the answer is easy for $n/2 < k \leq n$: the minimum number is
\begin{equation}
 \alpha(k, n) = 2(n-k)+1, \nonumber
\end{equation}
and it corresponds to a unique configuration. Next, the pattern starts getting more difficult at the ``middle point" $k = n/2$ or $(n-1)/2$, depending on the polarity of $n$. The easier case is when $n$ is even and $k = n/2$, the minimum number is found to be
\begin{equation}
\alpha(k, 2k) = 3k+1, \nonumber
\end{equation}
and uniqueness still holds. The harder case is when $n$ is odd and $k= (n-1)/2$. Uniqueness no longer holds. Three trials are attempted in order to write down a precise formula for $\alpha(k, 2k+1)$. The first trial results in the following upper bound: 
\begin{equation}
\alpha(k, 2k+1) \leq 4k+5. \nonumber
\end{equation}
The second trial, which hopes to fix the first one, results in the following upper bound:
\begin{equation}
\alpha(k, 2k+1) \leq \left\{\begin{array}{cc}
\frac{(7k+11)}{2} & \mbox{if $k$ is odd} \\
\frac{(7k+12)}{2} & \mbox{if $k$ is even.}
\end{array}
\right. \nonumber
\end{equation}
Moreover, the third trial, which hopes to fix the second one, again results only in an upper bound as follows:
\begin{equation}
\alpha(k, 2k+1) \leq \left\{ \begin{array}{cc}
		\frac{10k+24}{3} & \mbox{if $k \equiv 0$ (mod 3)}\\
		\frac{10k+23}{3} & \mbox{if $k \equiv 1$ (mod 3)}\\
		\frac{10k+25}{3} & \mbox{if $k \equiv 2$ (mod 3).}
	\end{array}
\right. \nonumber
\end{equation}
In the end, we find that there is an asymptotic upper bound for the minimum number: given any positive integer $a$, 
\begin{equation}
\alpha(k, 2k+1) \leq \frac{(3a+1)k+C}{a}, \nonumber
\end{equation}
where the constant $C$ depends on the congruence class of $k$ modulo $a$. So $\alpha(k, 2k+1)$ converges to $3k$ (plus a constant) as $k$ goes to infinity. 

Although there is no explicit formula for $\alpha(k, 2k+1)$, it sheds light on how the next minimum numbers $\alpha(k, 2k+2)$, $\alpha(k, 2k+3)$ etc.\ might change. Furthermore, it tells us that there is no surprise that the previous mathematicians could only give a sharp upper bound for this two-dimensional case, which is 
\begin{equation}
n^2 - \alpha(2, n) = O(n^{\frac{3}{2}}). \nonumber
\end{equation}
Finally, we will continue generalize our method to the three dimensional case and hope that the patterns we've discovered so far would be used to find an example that realizes the proposed sharp upper bound of $O(N^{11/4})$. Moreover, we ask a more general question of what is the minimum number of 1's one can put in an $n \times n \times n$ matrix with entries 0 and 1, such that every $k \times k \times k$ minor contains at least one entry of 1? What about to any arbitrary dimension?

\bibliographystyle{aomplain}

\begin{thebibliography}{22}
\bibitem{1}  A. Carbery, M. Christ, and J. Wright, \textit{Multidimensional Van-der-Corpus and sublevel set estimates}, J. Amer. Math Soc. 12 (1999), pp. 981-1016

\bibitem{2}  N. Katz, E. Krop, and M. Maggioni, \textit{Remarks on the Box Problem}, Mathematical Research Letters 9, 2002, pp. 515-519.
\end{thebibliography}

\end{document}